\documentclass[10pt]{amsart}
\usepackage[hmargin=40mm, height=200mm]{geometry}
\usepackage{amsmath,amssymb,eucal,mathrsfs,color,hyperref,amsthm,xspace}

\newcommand{\sH}{\mathscr{H}}
\newcommand{\sA}{\mathscr{A}} 

\newcommand{\R}{\mathbf{R}}
\newcommand{\C}{\mathbf{C}}
\newcommand{\bS}{\mathbf{S}}

\newcommand{\PSLR}{\mathrm{PSL}_{2}(\mathbf{R})}

\newcommand{\HH}{\mathbf{H}^{n}}
\newcommand{\HHI}{\mathbf{H}^{\infty}}

\newcommand{\Isomn}{{\rm Isom}(\HH)}
\newcommand{\Isomi}{{\rm Isom}(\HHI)}
\newcommand{\Isom}{{\rm Isom}}

\newcommand{\simil}{{\rm Sim}(\mathbf{R}^{\ell})}
\newcommand{\Out}{{\rm Out}}

\newcommand{\LLl}{{\rm L}^{2}(\mathbf{R}^{\ell})}

\newcommand{\LLB}{{\rm L}^{2}(\partial \HH)}
\newcommand{\corad}{{\rm corad}} 
\newcommand{\pinew}{\widetilde{\pi}}
\newcommand{\rhonew}{\widetilde{\varrho}}
\newcommand{\BNEW}{\widetilde{B}}

\newcommand{\se}{\subseteq}
\newcommand{\lra}{\longrightarrow}
\newcommand{\Id}{\mathrm{Id}}
\newcommand{\jac}{\mathrm{Jac}}
\newcommand{\cat}[1]{{\upshape CAT({\ensuremath#1})}\xspace}

\numberwithin{equation}{section}

\newtheorem{main}{Theorem}[section]
\newtheorem{prob}[main]{Problem}

\newtheorem{prop}[main]{Proposition}
\newtheorem{lemma}[main]{Lemma}
\newtheorem{rem}[main]{Remark}
\newtheorem{cor}[main]{Corollary}
\newtheorem{exam}[main]{Example}
\newtheorem{imain}{Theorem}

\newtheorem{iprop}[imain]{Proposition}
\theoremstyle{definition}
\newtheorem{nota}[main]{Notation}
\newtheorem{defi}[main]{Definition}

\title[An exotic deformation of the hyperbolic space]{An exotic deformation\\ of the hyperbolic space}
\author[Nicolas Monod]{Nicolas Monod$^*$}
\address{EPFL, Switzerland}
\email{nicolas.monod@epfl.ch}
\thanks{$^*$Supported in part by the Swiss National Science Foundation and the ERC}
\author[Pierre Py]{Pierre Py}
\address{IRMA, Universit\'e de Strasbourg \& CNRS, 67084 Strasbourg, France}
\email{ppy@math.unistra.fr}

\begin{document}

\begin{abstract}
On the one hand, we construct a continuous family of non-isometric proper \cat{-1} spaces on which the isometry group $\Isomn$ of the real hyperbolic $n$-space acts minimally and cocompactly. This provides the first examples of non-standard \cat{0} model spaces for simple Lie groups.

On the other hand, we classify all continuous non-elementary actions of ${\rm Isom}(\mathbf{H}^{n})$ on the infinite-dimensional real hyperbolic space. It turns out that they are in correspondence with the exotic model spaces that we construct.
\end{abstract}
\maketitle

\tableofcontents
\newpage


\section{Introduction}
\subsection{Overview}
Let $\Isomn \cong \mathbf{PO}(1,n)$ be the isometry group of the real hyperbolic space $\HH$ of dimension $n\geq 2$. The following theorem will be restated more precisely later in this introduction:

\begin{imain}\label{thm:bref}
There is a continuous family $\{C_t\}_{0<t\leq 1}$ of proper \cat{-1} spaces equipped with a continuous family of cocompact minimal isometric $\Isomn$-actions, abutting to $C_1=\HH$. The spaces $C_t$ are mutually non-isometric (even after rescaling) and satisfy ${\rm Isom}(C_t) =\Isomn$.
\end{imain}

This stands in sharp contrast to the fact that \emph{geodesically complete} \cat{0} model spaces for simple Lie groups are unique up to scaling, as proved in~\cite[Theorem~1.4]{capracemonod} (and thus answers Question~7.2 in~\cite{capracemonod2}).

\bigskip

The spaces $C_t$ of Theorem~\ref{thm:bref} will appear as canonical minimal invariant convex subsets in the \emph{infinite-dimensional} hyperbolic space $\HHI$ for a family of representations of $\Isomn$ on $\HHI$. These representations will be deformations of the standard embedding of $\Isomn$ into $\Isomi$ and likewise one can view $C_t$ as a deformation of the standard image $C_1=\HH\se \HHI$ arising as the limit of the composed standard embeddings $\mathbf{H}^{n}\se \mathbf{H}^{n+1}\se \mathbf{H}^{n+2} \ldots$

\medskip

Accordingly, a large part of this article is devoted to the study of representations of $\Isomn$ into $\Isomi$. We shall indeed classify the (continuous, isometric) actions of $\Isomn$ on $\HHI$.

We now pass to a more detailed presentation of the content of this article. 

\subsection{Context}
Consider a real Hilbert space $\mathscr{H}$ and a line $L$ in $\sH$. Let $B$ be the bilinear form on $\sH$ determined by the quadratic form 
$$B(v,v)=\vert v_{L}\vert^{2}-\vert v_{L^{\perp}}\vert^{2},$$
where $v_{L}$ and $v_{L^{\perp}}$ are the projections of $v$ on $L$ and $L^{\perp}$. We fix a non-zero vector $e\in L$. As in the finite-dimensional case, define the space $\HHI$ by the hyperboloid model:
$$\HHI:=\{v\in \mathscr{H}: B(v,v)=1, B(v,e)>0\}.$$
The formula 
$$\cosh(d_{\HHI}(x,y))=B(x,y)$$

\noindent defines a distance $d_{\HHI}$ on $\HHI$ which turns it into a complete \cat{-1} metric space. We will denote by $\partial \HHI$ the boundary of $\HHI$, which can be identified with the Grassmannian of $B$-isotropic lines in $\mathscr{H}$, and by $\Isomi$ its isometry group. The study of this space, as well as that of certain infinite-dimensional  symmetric spaces of finite rank, was suggested by Gromov in~\cite{gromov}. We refer the reader to~\cite{bim} for a detailed study of the space $\HHI$, and to~\cite{duchesne,duchesne2} for the case of higher rank symmetric spaces. This is the geometric viewpoint on the classical theory of \textit{Pontryagin spaces}~\cite{ioh,naimark1963,naimark2,naimark3,naimark4,pontryagin}.

\medskip

 An action $\varrho \colon  G \to \Isomi$ of a group $G$ on $\HHI$ is called elementary if it fixes a point in $\HHI \cup \partial \HHI$ or if it preserves a geodesic in $\HHI$. If $\varrho$ is non-elementary, one can prove that there exists a unique closed totally geodesic subspace $\HHI_{\varrho}$ of $\HHI$ which is $G$-invariant and minimal for this property (see~\cite{bim}). It is easily checked that an action $\varrho\colon G \to \Isomi$ is non-elementary and satisfies $\HHI_{\varrho}=\HHI$ if and only if the associated linear action of $G$ on $\mathscr{H}$ is irreducible. In this case, we will simply say that the action of $G$ on $\HHI$ is irreducible.  

\medskip

 It is a classical fact that any continuous isometric action of the group $\Isomn$ on a {\it finite-dimensional} symmetric space of non-compact type $X$ preserves the image of a totally geodesic embedding $\HH \hookrightarrow X$. This is a particular case of a theorem of Karpelevich and Mostow~\cite{karpe,mostow} (see also~\cite{gp}, or~\cite{boubelzeghib} for a geometric proof in the hyperbolic case). The starting point of this article is the following fact, showing that the Karpelevich--Mostow theorem does not hold anymore in the infinite-dimensional context:

\smallskip
\begin{center}
\begin{minipage}[t]{0.85\linewidth}
\itshape
There exist continuous irreducible isometric actions of $\Isomn$ on the space $\HHI$.
\end{minipage}
\end{center}

\smallskip
\noindent In other words, there exist irreducible representations of $\Isomn$ on a Hilbert space which are not unitary but which instead preserve a quadratic form of index $1$. More generally, there also exist irreducible representations of the group $\Isomn$ on a Hilbert space which preserve a quadratic form of index $p$, for certain values of $p>1$ (see Section~\ref{higherrank}). These facts are well-known to representation theorists and can be found, for instance, in~\cite{johnsonwallach}. 

These representations were put in a geometric context in~\cite{delzantpy} where the authors showed how actions of $\PSLR$ on $\HHI$ arise in the study of fundamental groups of compact K\"{a}hler manifolds. There are other examples of groups acting on an infinite-dimensional hyperbolic space: for the Cremona group, see~\cite{cantat}, for automorphism groups of trees, see~\cite[Chap.~6]{gromov} and~\cite{bim}.

\subsection{Results}

If $\varrho \colon  G \to \Isomi$ is an action of a group $G$ on $\HHI$, we will denote by $\ell_{\varrho}$ the associated {\it translation length function}, i.e. the function defined by $$\ell_{\varrho}(g)=\underset{x\in \HHI}{{\rm inf}}d_{\HHI}(\varrho(g)(x),x).$$ \noindent We will denote by $\ell_{\HH}$ the translation length function for the standard action of $\Isomn$ on $\HH$. We can now state our first main result, which classifies the (continuous) irreducible representations $\Isomn \to \Isomi$.

\begin{imain}\label{class}
\ 
\begin{enumerate}
\item Let $\varrho \colon  \Isomn \to \Isomi$ be a continuous non-elementary action. There exists $t\in (0,1]$ such that $\ell_{\varrho}(g)=t\ell_{\HH}(g)$ for all $g\in \Isomn$. Moreover, $t=1$ if and only if $\varrho$ preserves an $n$-dimensional totally geodesic subspace of $\HHI$.\label{class:exist}
\item For each $t\in (0,1)$ there is, up to conjugacy in $\Isomi$, exactly one irreducible continuous representation $\varrho_{t} \colon  \Isomn \to \Isomi$ such that $\ell_{\varrho_{t}}=t\ell_{\HH}$.
\end{enumerate}
\end{imain}

Thus, all irreducible continuous representations $\varrho \colon  \Isomn \to \Isomi$ arise from an explicit one-parameter family of representations. It could be interesting to search for another proof of this using the machinery of Lie algebra representations~\cite{wallach2}.

\medskip

In what follows, we will denote by $g_{hyp}$ the Riemannian metric with constant curvature $-1$ both on $\HH$ and on $\HHI$. We also denote by $d_{\HH}$ and $d_{\HHI}$ the associated distance functions on $\HH$ and $\HHI$ respectively and by $\overline{\HHI}$ (resp. $\overline{\HH}$) the geometric bordification $\HHI \cup \partial \HHI$ (resp. $\HH \cup \partial \HH$), endowed with the cone topology~\cite[II.8]{bh}. 

\begin{imain}\label{curv} There exists a smooth harmonic $\varrho_{t}$-equivariant map $f_{t} \colon  \HH \to \HHI$ with the following properties. 
\begin{enumerate}
\item It is asymptotically isometric after rescaling, i.e.\ there is  $D\ge 0$ such that
$$\big| d_{\HHI}(f_{t}(x),f_{t}(y))\ - td_{\HH}(x,y)\big|\le D \kern 15mm(\forall\,x, y\in \HH).$$\label{curv:qi}
\item The image of $f_{t}$ is a minimal submanifold of curvature $\frac{-n}{t(t+n-1)}$ and we have
$$f_{t}^{\ast}\,g_{hyp}\ =\ \frac{t(t+n-1)}{n}\,g_{hyp}.$$\label{curv:curv}
\item The map $f_{t}$ extends to a continuous map from $\overline{\HH}$ to $\overline{\HHI}$.\label{curv:bd}
\end{enumerate}
\end{imain}

\noindent
In fact, it will be clear from the model for $\varrho_t$ that $f_t$ is the only equivariant map $\HH \to \HHI$ (Lemma~\ref{lemma:unique:k} below). When $n=2$, the existence of the family $\varrho_{t}$ and some of its properties were established in~\cite{delzantpy}. 

\bigskip
Before describing our next result, let us recall a theorem from~\cite{capracemonod}. Let $k$ be a local field and ${\bf G}$ be an absolutely almost simple simply connected $k$-group. Let $X_{\rm model}$ be the Riemannian symmetric space or Bruhat-Tits building associated with ${\bf G}(k)$. The following is proved in~\cite{capracemonod} (see Theorem~7.4 in~\cite{capracemonod} for a more precise statement): 

\smallskip
\begin{center}
\begin{minipage}[t]{0.85\linewidth}
\itshape
Let $X$ be a non-compact \cat0 space on which ${\bf G}(k)$ acts continuously and cocompactly by isometries. If $X$ is geodesically complete, then $X$ is isometric to $X_{\rm model}$ (possibly rescaled).
\end{minipage}
\end{center}

\medskip
 Without the assumption of geodesic completeness, it is known that this result can fail when $k$ is non-Archimedian. More precisely, $X$ need not contain a closed, $\mathbf{G}(k)$-invariant convex set isometric to $X_{\rm model}$ (see Example~7.6 in~\cite{capracemonod}). When $k$ is Archimedian, it was not known whether the theorem above remains true without the hypothesis of geodesic completeness (see Question~7.2 in~\cite{capracemonod2}). The next theorem shows that this is not the case, at least for the group $\Isomn$. Indeed, we shall construct exotic cocompact spaces $X$, which are even nice enough to be \cat{-1} and without branching geodesics, since they arise as convex subspaces of $\HHI$.

\smallskip
It is known that any non-elementary isometric group action on $\HHI$ admits a unique minimal non-empty closed convex invariant subspace $C\se\HHI$ (see Section~\ref{sec:Ct:def} below). In the case of the action $\varrho_t$ ($t\in (0,1)$), we denote it by $C_t$ and it turns out to provide our ``exotic'' space. In what follows, we will use the following convention: $C_{1}$ will be the hyperbolic space $\HH$ and $f_{1}$ the identity map $\HH \to C_{1}$. 

\begin{imain}\label{convexhull}
For any $t\in (0,1]$, the  \cat{-1} space $C_{t}$ is locally compact and the action of $\Isomn$ on $C_{t}$ is cocompact.

Moreover, the spaces $C_{t}$ and $C_{t'}$ are isometric only if $t=t'$ (even upon rescaling the metric).
\end{imain}

The following proposition describes $C_{t}$ more precisely.

\begin{iprop}\label{propertiesconvexhull} Let $t\in (0,1]$.
\begin{enumerate}
\item The set $C_{t}$ coincides with the closed convex hull of $f_{t}(\HH)$.
\item The set $C_{t}$ coincides with the closed convex hull of $f_{t}(\partial \HH)$ and $\partial C_{t}= f_{t}(\partial \HH)$.
\item $\Isom(C_t) = \Isomn$.
\end{enumerate}
\end{iprop}

\medskip

Next, we study the continuity with respect to $t$ of the metric space $C_{t}$. Recall that there is a natural topology called the {\it pointed Gromov--Hausdorff topology} on the class of all pointed metric spaces. We refer the reader to~\cite[chap.~3]{gromov0} or to Section~\ref{sec:GH} for its definition. If $G$ is a fixed topological group, a metric space endowed with a continuous isometric action of the group $G$ will be called, for short, a {\it metric $G$-space}. On the class of all pointed metric $G$-spaces, there is a natural refinement of the pointed Gromov--Hausdorff topology which takes into account the $G$-action (see for instance~\cite{fukaya}). We will define precisely a suitable version of this refined topology in Section~\ref{sec:GH} and call it the {\it strong topology} on pointed metric $G$-spaces (rather than the ``equivariant pointed Gromov--Hausdorff topology"\ldots).

The spaces $C_{t}$ will always be pointed at $f_t(o)$, where $o$ is a given point of $\HH$. 

\begin{imain}\label{continuous}
The map $t \mapsto C_t$ is continuous on $(0,1]$ for the strong topology on the class of pointed $\Isomn$-spaces. In particular, as $t\to 1$, $C_t$ converges to $\HH$ and the diameter of $C_t/\Isomn$ goes to zero.
\end{imain}

Concerning the behavior of the spaces $C_{t}$ as $t$ goes to $0$, we refer the reader to Section~\ref{renor} for a discussion.

\bigskip

The article is organized as follows. In Section~\ref{unique}, we prove completely Theorem~\ref{class} except for the existence of the family $\varrho_{t}$.  In Section~\ref{princeseries}, we recall some classical facts concerning the spherical principal series of representations of $\Isomn$ and study the actions of $\Isomn$ on $\HHI$ which arise from the spherical principal series. We then complete the proof of Theorem~\ref{class} and prove Theorem~\ref{curv}. In Section~\ref{sec:Ct}, we study the spaces $C_{t}$: we prove Theorem~\ref{convexhull}, Proposition~\ref{propertiesconvexhull} as well as Theorem~\ref{continuous}. Finally, Section~\ref{further} contains further results: Section~\ref{furthertree} deals with actions of automorphism groups of trees on $\HHI$, already studied in~\cite{bim} and Section~\ref{higherrank} starts the study of actions of $\Isomn$ on certain higher rank infinite-dimensional symmetric spaces which also arise from the spherical principal series.  

\medskip

\noindent {\bf Acknowledgments.} We would like to thank Serge Cantat, Thomas Delzant and Robert Stanton for several useful conversations during the preparation of this work.  We are very grateful to the anonymous referee for many useful remarks which improved the presentation of the text.


\section{Classification of non-elementary actions}\label{unique}

In this section, we start the classification of non-elementary continuous actions of the group $\Isomn$ on $\HHI$. Here {\it continuous} means that all orbit maps $\Isomn \to \HHI$ are continuous. The strategy of the proof is similar to the one in~\cite{bim} where an analogous result is proved for actions of the automorphism group of a regular tree on $\HHI$: we first prove that the action is determined by its restriction to a parabolic subgroup and then study the action of parabolic subgroups. The second step of the proof differs from the case of trees since parabolic subgroups have distinct structures in $\Isomn$ and in the automorphism group of a tree. 

\medskip
We recall that isometries of general complete \cat0 spaces are subdivided in three types according to the behavior of their displacement function:  elliptic, parabolic or hyperbolic; see~\cite[II.6]{bh}. In the particular case of $\HH$ and $\HHI$, the stronger \cat{-1} condition implies that hyperbolic isometries admit a unique axis and that parabolic isometries have a unique fixed point at infinity; see~\cite[\S4]{bim} (the existence of the fixed point is not immediate for non-proper spaces).

\subsection{Restriction to a parabolic subgroup}\label{parabolic}

We start this section by introducing some notation. We consider the vector space $\R^{n+1}$ endowed with the bilinear form $B$ defined by:
$$B(x,x)=x_{1}^{2}-x_{2}^{2}-\cdots -x_{n+1}^{2}.$$

\noindent As usual, the hyperbolic space $\HH$ is described as:
$$\HH:=\{x\in \R^{n+1}, B(x,x)=1, x_{1}>0\};$$
the group $\Isomn$ is the group of linear maps $g \colon  \R^{n+1} \to \R^{n+1}$ which preserve $B$ and satisfy $g(\HH)=\HH$. If $\xi$ is a $B$-isotropic vector in $\R^{n+1}$, we denote by $[\xi]$ the corresponding point of $\partial \HH$.

We denote by $e_{1}, \ldots e_{n+1}$ the canonical basis of $\R^{n+1}$ and by $\xi_{1}=\frac{1}{\sqrt{2}}(1,1,0,\ldots)$ and $\xi_{2}=\frac{1}{\sqrt{2}}(1,-1,0,\ldots)$. The vectors $\xi_{1}$ and $\xi_{2}$ are $B$-isotropic. In what follows we will rather use the basis $(\xi_{1}, \xi_{2}, e_{3}, \ldots)$. In particular all the matrices we write below are decomposed into blocks corresponding to the decomposition
$$\R^{n+1}=\R\xi_{1} \oplus \R \xi_{2} \oplus U,$$

\noindent where $U$ is the vector space generated by the family $(e_{i})_{3\le i \le n+1}$, which we will identify with $\R^{n-1}$. Hence they have the following form: 
$$\left( \begin{array}{ccc}
\alpha & \beta & l_{1}\\
\gamma & \delta & l_{2}\\
v_{1} & v_{2} & M \\
\end{array}\right),$$ \noindent where $\alpha, \beta, \gamma, \delta$ are real numbers, $v_{1}$ and $v_{2}$ are vectors in $\R^{n-1}$, $l_{1}$ and $l_{2}$ are linear forms on $\R^{n-1}$ and $M$ is an endomorphism of $\R^{n-1}$.

If $\lambda \in \R^{\ast}_{+}$, $A\in {\rm O}(n-1)$ and $v\in \R^{n-1}$, let $g_{\lambda,v,A}$ be the element of ${\rm Isom}(\mathbf{H}^{n})$ whose matrix has the following form: 

$$\left( \begin{array}{ccc}
\lambda & \frac{1}{2}\lambda \vert v \vert^{2}& \lambda \langle v, A(\cdot) \rangle\\
0 & \lambda^{-1} & 0 \\
0 & v & A\\
\end{array}\right).$$

\noindent Here $\langle \cdot , \cdot \rangle$ stands for the standard inner product on $\R^{n-1}$. The group $P=\{g_{\lambda,v,A}\}$ is the stabilizer of the line $\R \xi_{1}$ in $\Isomn$. It is isomorphic to the group ${\rm Sim}(\R^{n-1})$ of similarities of $\R^{n-1}$: the map which takes $g_{\lambda,v,A}$ to the similarity
$$x\longmapsto \lambda Ax+\lambda v$$

\noindent is an isomorphism from $P$ to ${\rm Sim}(\R^{n-1})$.

Note that the isometry $g_{\lambda,v,A}$ is parabolic if and only if $\lambda=1$ and $v\notin {\rm Im}(A-\Id)$. Hence a parabolic element $g_{1,v,A}$ is always conjugated inside $P$ to an element of the form 
\begin{equation}\label{matriceparab}
\left(\begin{array}{cccc}
1 & \frac{1}{2}\vert v'\vert^{2} & \langle v',\cdot \rangle & 0\\
0 & 1 & 0 & 0\\
0 & v' & \Id & 0\\
0 & 0 & 0  & A'\\ 
\end{array}\right),
\end{equation}

\noindent with $A'-\Id$ invertible and $v'\neq 0$. It follows that if $g_{1,v,A}$ is parabolic, there is always another parabolic element $h\in P$ which commutes with $g_{1,v,A}$ and which is conjugated in $P$ to its own square: if $g_{1,v,A}$ is equal (up to conjugacy) to the matrix in equation~\eqref{matriceparab}, one can take
\begin{equation}
h=\left(\begin{array}{cccc}
1 & \frac{1}{2}\vert v'\vert^{2} & \langle v',\cdot \rangle & 0\\
0 & 1 & 0 & 0\\
0 & v' & \Id & 0\\
0 & 0 & 0  & \Id\\ 
\end{array}\right)
\end{equation}
\noindent
and the conjugating element is the diagonal matrix with entries $(2, 1/2, \Id, \Id)$.

Finally, we will denote by $\sigma$ the involution represented by the matrix 
\begin{equation}\label{involution}
\sigma=\left( \begin{array}{ccc}
0 & 1 & 0\\
1 & 0 & 0\\
0 & 0 & J_{1}\\
\end{array}\right),
\end{equation}
where $J_{1}$ is the diagonal matrix with coefficients $(-1,1, \ldots )$. The group $\Isomn$ is generated by $\sigma$ and $P$. 

\bigskip
From now on, we fix a continuous non-elementary representation
$$\varrho \colon  \Isomn \lra\Isomi.$$ 
Let $x_{0}$ be a point in $\HHI$. Since $\varrho(\Isomn)$ does not fix any point of $\HHI$ and thanks to a classical argument going back to Cartan (see~\cite[II.2]{bh}), the function
$$ \varphi_{\varrho,x_{0}}\colon G\longrightarrow \R_+, \kern10mm \varphi_{\varrho,x_{0}}(g)=d_{\HHI}(\varrho(g)x_{0},x_{0})$$
\noindent has to be unbounded. It is then a well-known result that this implies that the function $\varphi_{\varrho, x_{0}}$ is \emph{proper}. The reader will find a proof of this fact in~\cite{decor} in a more general context. (For Hilbert spaces, the properness of the displacement function for actions of simple Lie groups of rank~$1$ was proved by Shalom~\cite{shalom}.) The properness of $\varphi_{\varrho, x_{0}}$ enters the proof of the following:

\begin{prop}\label{type}
The action $\varrho$ preserves the type, i.e. the image $\varrho(g)$ of an elliptic element $g$ of $\Isomn$ (resp. parabolic, hyperbolic) is an elliptic isometry of $\HHI$ (resp. parabolic, hyperbolic).  Moreover, $\varrho(P)$ has a unique fixed point in $\partial \HHI$.
\end{prop}

\begin{proof}
Since elliptic elements of $\Isomn$ are contained in a compact group, their image under $\varrho$ has to be elliptic. Since the function $\varphi_{\varrho, x_{0}}$ is proper, the image by $\varrho$ of a non-elliptic element in $\Isomn$ has to be non-elliptic. We start with the case where $g\in \Isomn$ is parabolic; as noted above, there exists another parabolic element $h$ which commutes with $g$ and which is conjugated to its own square. Suppose for a contradiction that $\varrho(g)$ is hyperbolic; then $h$ must preserve the unique axis of $g$. This shows that $\varrho(h)$ is non-parabolic; being non-elliptic, it must be hyperbolic, contradicting the fact that it is conjugated to its square.

From this, one already sees that the group $\varrho(P)$ has a unique fixed point $[\eta_{1}]$ in the boundary of $\HHI$. Indeed the elements $\varrho(g_{1,v,\Id})$ for $v\in \R^{n-1}$ are commuting parabolic isometries, hence have a common unique fixed point $[\eta_{1}]$. Since the group formed by the $g_{1,v,\Id}$ is normal in $P$, $[\eta_{1}]$ must be fixed by all of $\varrho(P)$. 

Finally, let $g\in \Isomn$ be hyperbolic. Up to conjugacy and up to replacing $g$ by a power, we can assume that $g$ lies on a $1$-parameter subgroup of the form $g_{t}=g_{e^{ct},0,A_{t}}$. If $g$ were parabolic, the entire $1$-parameter subgroup $g_{t}$ would have a unique fixed point $\xi$ in $\partial \HHI$. Note that the point $\xi$ would coincide with the fixed point at infinity $[\eta_{1}]$ for the whole group $\varrho(P)$. In particular the $1$-parameter group $g_{e^{t},0,\Id}$ would admit $\xi$ as a unique fixed point. Since the involution $\sigma$ normalizes the latter group; $\xi$ would be fixed by $\sigma$. Finally, being fixed by $\sigma$ and $P$, $\xi$ would be fixed by all of $\Isomn$. This is a contradiction since $\varrho$ is non-elementary.
\end{proof}

We now continue our study of the action of $P$ and still denote by $[\eta_{1}]$ the unique fixed point of $\varrho(P)$ in $\partial \HHI$. Since $\varrho(P)$ does not preserve the horospheres centered at $[\eta_{1}]$ (it contains hyperbolic isometries), it follows from Proposition~4.3 in~\cite{bim} that there exists a unique closed totally geodesic subspace $\HHI_{P}$ which is $P$-invariant and minimal with this property. We can actually describe this space. Let $\eta_{2}$ be an isotropic vector such that the lines $[\eta_{1}]$ and $[\eta_{2}]$ are the two fixed points of the hyperbolic isometries $\varrho(g_{\lambda,0,\Id})$. Assume that $B(\eta_{1},\eta_{2})=1$. Let $E=\eta_{1}^{\perp}\cap \eta_{2}^{\perp}$. We will represent the isometries $\varrho(g)$ for $g\in P$ by matrices associated to the decomposition $\mathscr{H}=\R\eta_{1}\oplus \R \eta_{2}\oplus E$. We also denote by $\langle\cdot , \cdot \rangle$ the scalar product induced by $-B$ on $E$. One can then write:
$$\varrho(g_{\lambda,v,A})=\left(\begin{array}{ccc}
\chi(\lambda) & \frac{1}{2}\chi(\lambda)\vert c(\lambda,v,A)\vert^{2} & \chi(\lambda)\langle c(\lambda,v,A),\pi(\lambda,v,A)(\cdot)\rangle \\
0 & \chi(\lambda^{-1}) & 0\\
0 & c(\lambda,v,A) & \pi(\lambda,v,A)\\ 
\end{array}\right)$$
\noindent where $\chi \colon  \R_{+}^{\ast}\to \R^{\ast}_{+}$ is a continuous homomorphism, $\pi$ is an orthogonal representation of $P$ in $E$ and the continuous map
$$c \colon  P \lra E$$
\noindent
satisfies that $u(\lambda,v,A):=\chi(\lambda)c(\lambda,v,A)$ is a cocycle for the linear representation $\chi \otimes \pi$. Explicitly, this means the following: keeping in mind that
$$g_{\lambda,v,A}\cdot g_{\lambda',v',A'} = g_{\lambda\lambda' ,\lambda'^{-1} v + Av',A A'},$$
we have
\begin{equation}\label{eq:cocycle}
c(\lambda\lambda' ,\lambda'^{-1} v + Av',A A') = \chi(\lambda'^{-1}) c(\lambda,v,A) + \pi(\lambda,v,A) c(\lambda',v',A').
\end{equation}

\medskip
Let
$$V:=\overline{Span\{c(g), g\in P\}}$$
be the closed vector space spanned by the vectors $c(g)$ for $g\in P$. Then it is easy to see that $\R\eta_{1}\oplus \R\eta_{2}\oplus V$ is $\varrho(P)$-invariant and minimal with this property. The space $\HHI_{P}$ is thus the intersection of $\R\eta_{1}\oplus \R\eta_{2}\oplus V$ with $\HHI$. The map $c$ enjoys more properties. The fact that $g_{\lambda,0,\Id}$ and $g_{\mu,0,A}$ commute implies that $\varrho(g_{\mu,0,A})$ preserves the axis of $\varrho(g_{\lambda,0,\Id})$, we thus have:
\begin{equation}\label{eq:c_vanish}
c(\mu,0,A)=0.
\end{equation}
A simple computation with~\eqref{eq:cocycle} shows that this implies
\begin{equation}\label{aaaaa}
\chi (\lambda) c(\lambda,v,A)=c(1,\lambda v, \Id).
\end{equation} 
Hence we conclude:

\begin{lemma}\label{lem:cyclic}
The space $V$ is the closure of the vector space generated by the $c(1,v,\Id)$ for $v\in \R^{n-1}$.\qed
\end{lemma}

\begin{prop}
One has $\chi(\lambda)=\lambda^{t}$ for some $0< t  \le 1$. Moreover, if $ t =1$, then $P$ preserves an $n$-dimensional totally geodesic subspace of $\HHI$. 
\end{prop}

\begin{proof}
Consider the continuous map $f\colon \R^{n-1}\to E$ defined by $f(v):=c(1,v,\Id)$. Notice that $f(v)\neq 0$ when $v\neq 0$ since $c$ is non-zero for parabolic isometries. For any $\lambda>0$ we have
$$g_{1, \lambda v, \Id} = g_{\lambda, 0, \Id}\cdot g_{1, v, \Id}\cdot g_{\lambda^{-1}, 0, \Id}$$
and hence, using twice~\eqref{eq:cocycle} and~\eqref{eq:c_vanish} we conclude
\begin{equation}\label{eq:scale}
f(\lambda v) = \chi(\lambda) \pi(g_{\lambda, 0, \Id}) f(v), \kern5mm\text{hence}\kern3mm \vert f(\lambda v) \vert = \chi(\lambda) \vert f(v)  \vert.
\end{equation}
Let us write $\chi(\lambda)=\lambda^{t}$ for some non-zero real number $t$. Then~\eqref{eq:scale} together with the continuity of $f$ at the origin already implies $t>0$. On the other hand, the relation~\eqref{eq:cocycle} also gives
\begin{equation}\label{ad} 
f(v+w)=f(v)+\pi(g_{1,v,\Id}) f(w).
\end{equation}
The triangle inequality implies $\vert f(2 v)\vert \le  2  \vert f(v)\vert$; together with~\eqref{eq:scale} this implies $t \le 1$. If $t=1$, we must have $\pi(g_{1,v,\Id}) f(v)=f(v)$. We shall write from now on $S_{v}:=\pi(g_{1,v,\Id})$ to simplify notation. Note that $S_{v+w}=S_{v}\circ S_{w}$. According to the previous remark we have:

$$S_{v_{1}+v_{2}}(f(v_{1}+v_{2}))=f(v_{1}+v_{2}).$$
\noindent Using the cocycle relation~(\ref{ad}) on $f$ we deduce:
$$S_{v_{1}}S_{v_{2}}(f(v_{1})+S_{v_{1}}f(v_{2}))=f(v_{2})+S_{v_{2}}(f(v_{1})),$$
\noindent and then
$$S_{2v_{1}}(f(v_{2}))=f(v_{2}).$$
\noindent Hence $S_{w}(f(v))=f(v)$ for all $v,w,\in \R^{n-1}$. Using Equation~(\ref{ad}) again, this implies that $f\colon \R^{n-1} \to E$ is linear. It has to be injective; indeed if $f(v)=0$ and $v\neq 0$, the isometry $\varrho(g_{1,v,\Id})$ is elliptic, in contradiction with the fact that $\varrho$ preserves the type. Hence $V$ has dimension $n-1$, and $\varrho$ preserves the linear space $\R\eta_{1}\oplus \R\eta_{2}\oplus V$.
\end{proof}

Note that if $V$ contains a non-trivial subspace on which $\R^{n-1}\simeq \{g_{1,v,\Id}\}$ acts trivially we must have $t=1$. Indeed let  $p$ be the orthogonal projection from $V$ to the space of fixed points of $\R^{n-1}$. Then $p\circ f$ is a homomorphism, which is non-trivial since the image of $f$ generates a dense subspace of $V$. This implies that $t=1$. Hence, if $t<1$, the orthogonal representation $\pi$ does not have $\R^{n-1}$-invariant vectors inside $V$. 

\medskip
To prove the first part of Theorem~\ref{class}, it suffices to show that the translation length function associated to $\varrho$ satisfies $\ell_{\varrho}=t\ell_{\HH}$ for hyperbolic elements of $\Isomn$ since we have already established that the action $\varrho$ preserves the type. It is enough to consider $P$ since every hyperbolic element is conjugated to an element of $P$. Now the statement follows from the previous proposition because the  translation length is the logarithm of the Busemann character at the forward fixed point (both in $\HH$ and in $\HHI$) and the Busemann character of $g_{\lambda,v,A}$ is $\lambda$ in $\HH$ and $\chi(\lambda)$ in $\HHI$.

\medskip
From now on, we will assume that the action $\varrho$ is not only non-elementary but also irreducible. We will prove that $\varrho$ is completely determined by the number $t$ appearing in the previous proposition. 

\begin{prop}
The subspace $V$ is equal to $E$ and the representation $\varrho$ is determined by its restriction to $P$. 
\end{prop}

\begin{proof}
We have already noticed that $P$ together with the involution $\sigma$ generate the group $\Isomn$. Here, as in~\cite{bim}, we use the relations between $\sigma$ and $P$ to prove that $\varrho(\sigma)$ is determined by the restriction of $\varrho$ to $P$. More precisely, we consider the product
$$g:=\sigma \cdot g_{\lambda,v,\Id}\cdot \sigma \cdot g_{\mu,w,\Id}\cdot g_{\lambda,v,\Id}\cdot \sigma,$$ 
where $\lambda, \mu \in \R^{\ast}$ and $v\in \R^{n-1}- \{0\}$ are fixed and $w$ will be determined. As in the definition of the involution $\sigma$, we will denote by $J_{1}$ the orthogonal reflexion of $\R^{n-1}$ fixing the hyperplane orthogonal to the first vector of the canonical basis. A long but simple calculation shows that if $w$ is chosen so that
$$\frac{w}{\lambda}+v=\frac{-2J_{1}(v)}{\lambda \mu \vert v \vert^{2}},$$
then the element $g$ above belongs to $P$ and one actually has $g=g_{\eta,u,A}$ for
$$\eta=\frac{2}{\lambda^{2}\mu \vert v \vert^{2}}, \kern10mm u=\lambda \mu J_{1}(v), \kern 10mm A=s_{u}\circ J_{1}.$$
Here $s_{u}$ is the reflection orthogonal to the vector $u$. We now use the relation 
$$\varrho(g)=\varrho(\sigma) \varrho(g_{\lambda,v,\Id}) \varrho(\sigma) \varrho(g_{\mu,w,\Id})\varrho(g_{\lambda,v,\Id}) \varrho(\sigma)$$
to determine $\varrho(\sigma)$. Since $\varrho(\sigma)$ interchanges the two fixed points at infinity of $\varrho(g_{\lambda,0,\Id})$, we can write:
$$\varrho(\sigma)=\left(\begin{array}{ccc}
0 & \nu^{-1} & 0\\
\nu & 0 & 0\\
0 & 0 & \pi (\sigma)\\
\end{array}\right).$$ We first compute the product $\varrho(\sigma)\varrho(g_{\lambda,v,\Id})\varrho(\sigma)\varrho(g_{\mu,w,\Id})\varrho(g_{\lambda,v,\Id})$ and find that it has the following form:
$$\left( \begin{array}{ccc}
\chi(\mu) & \frac{\chi (\mu)}{2}\vert c(\lambda \mu , \frac{w}{\lambda}+v,{\rm Id})\vert^{2} & \ast \\
\frac{1}{2}\nu^{2}\chi(\lambda^{2}\mu)\vert c(\lambda,v,\Id)\vert^{2} & \ast & \ast\\
 \nu \chi(\lambda \mu)\pi(\sigma)(c(\lambda,v,\Id)) & \ast & \ast\\
\end{array}\right).$$
\noindent Hence $\varrho(g)$, which is obtained from the matrix above by multiplying on the right by $\varrho(\sigma)$, reads:
$$\left(\begin{array}{ccc}
\frac{1}{2}\nu \chi(\mu)\vert c(\lambda\mu,\frac{w}{\lambda}+v,\Id)\vert^{2} & \ast & \ast \\
\ast & \ast & \ast \\
\ast & \chi( \lambda \mu)\pi(\sigma)(c(\lambda,v,\Id))& \ast \\
\end{array}\right).$$

\noindent But this must be equal to the matrix
$$\left( \begin{array}{ccc}
\chi(\eta) & \frac{1}{2}\chi(\eta)\vert c(\eta,u,A)\vert^{2} & \chi(\eta)\langle c(\eta,u,A),\pi(\eta,u,A)(\cdot)\rangle \\
0 & \chi(\eta^{-1}) & 0\\
0 & c(\eta,u,A) & \pi(\eta,u,A)\\ 
\end{array}\right)$$

\noindent This shows that $\frac{1}{2}\nu \chi(\mu)\vert c(\lambda\mu,\frac{w}{\lambda}+v,\Id)\vert^{2}=\chi(\eta)=\chi(2/\lambda^{2}\mu\vert v \vert^{2})$, hence $\nu$ is determined by the restriction of $\varrho$ to $P$. It further shows
$$\pi(\sigma)(c(\lambda,v,\Id))=\chi(\lambda \mu)^{-1}c(\eta,u,A).$$
Since the vectors $c(\lambda,v,\Id)$ generate a dense subspace of $V$, this implies that $\pi(\sigma)$ preserves the space $V$, hence $\varrho(\sigma)$ preserves the space $\R \eta_{1}\oplus \R \eta_{2} \oplus V$. Since $\varrho$ was assumed irreducible, this implies that 
$$\mathscr{H} = \R \eta_{1} \oplus \R \eta_{2} \oplus V,$$
\noindent in particular $V=E$. Moreover $\pi (\sigma)$ is determined on a dense subspace of $E$, hence everywhere on $E$. This completes the proof of the proposition.
\end{proof}

We must now determine completely the restriction of the representation $\varrho$ to $P$, from the number $t$ only. Observe that the stabilizer of $[\eta_{1}]$ in $\Isomi$ is isomorphic to the group ${\rm Sim}(E)$ of similarities of the Hilbert space $E$. Hence, we must study homomorphisms
$$P \lra {\rm Sim}(E),$$
subject to the various constraints we just collected. We will do this in the next paragraph.

\subsection{\texorpdfstring{Affine representations of the group $\simil$}{Affine representations}}
In this section, we work directly with the group $\simil$ of similarities of $\R^{\ell}$. We denote by $S_{\lambda,v,A}$ the map
$$\begin{array}{rcl}
\R^{\ell} & \lra & \R^{\ell}\\
x & \longmapsto & \lambda A(x)+v\\
\end{array}$$
where, once again, $\lambda \in \R_{+}^{\ast}$, $A\in {\rm O}(\ell)$ and $v\in \R^{\ell}$. We will study representations (subject to various constraints) of $\simil$ in the group of similarities of a Hilbert space.

We thus start with a representation $\alpha_{t} \colon  \simil \to {\rm Sim}(E)$, where ${\rm Sim}(E)$ is the group of similarities of a real Hilbert space $E$. We write
$$\alpha_{t}(S_{\lambda,v,A})(x)=\lambda^{t}\pi(S_{\lambda,v,A})(x)+u(S_{\lambda,v,A})\;\;\;\;\; (x\in E)$$
with $\pi$ orthogonal. Thus $\pi$ is an orthogonal representation and $u$ is a $1$-cocycle for the representation $\pi\otimes \chi_{t}$, where $\chi_{t}$ is the character of $\simil$ defined by $\chi_{t}(S_{\lambda,v,A})=\lambda^{t}$. We shall make the following standing assumptions in view of the results of the previous section:

\smallskip
\begin{enumerate}
\item $0<t<1$,
\item $\alpha_{t}$ preserves the type (defined via the embedding ${\rm Sim}(E)\to\Isomi$),
\item $\pi$ has no $\R^{\ell}$-invariant vectors,\label{cond:no-inv}
\item the vector space generated by the vectors $(u(S_{1,v,\Id}))_{v\in \R^{\ell}}$ is dense in $E$,\label{cond:dense}
\item $u(S_{\lambda,v,A})=u(S_{1,v,\Id})$.\label{cond:conj}
\end{enumerate}
We will see in the proof of Lemma~\ref{dimco} that the last condition can always be achieved up to conjugacy. Note also that with $\ell=n-1$, the last equation is equivalent to identity~(\ref{aaaaa}), using the natural isomorphism
\begin{equation}\label{isooooo}
\begin{array}{rcl}
P & \lra & {\rm Sim}(\R^{\ell})\\
g_{\lambda,v,A} & \longmapsto & S_{\lambda,\lambda v, A}\\
\end{array}
\end{equation}
between the group $P$ considered in the previous section and the group ${\rm Sim}(\R^{\ell})$.

\bigskip

The purpose of this section is to prove that a representation $\alpha_{t}$ satisfying the conditions above is completely determined by $t$. Applying this result with $\ell=n-1$ allows to conclude the proof of Section~\ref{parabolic}: indeed, using the isomorphism~\eqref{isooooo} we obtain that the representation $P\to {\rm Sim}(E)$ appearing at the end of the previous section (and hence the original representation $\varrho \colon \Isomn \to \Isomi$) is completely determined by the parameter $t$, thus completing the proof of Theorem~\ref{class}, up to the existence of the representations $\varrho_{t}$.

\medskip
 We will denote by $\pi_{0}$ the unitary representation of $\simil$ on $\LLl$ defined by
 $$\big(\pi_{0} \left( S_{\lambda,v,A}\right) f\big) (y)=\lambda^{\frac{\ell}{2}}e^{i\langle y,v\rangle}f(\lambda A^{-1}y).$$
 We will need the following information about $\pi_{0}$: 
 \begin{lemma}\label{dimco} The space $${\rm H}^{1}(\simil, \chi_{t}\otimes \pi_{0})$$ 
\noindent is $1$-dimensional, generated by the cohomology class of the cocycle defined by
$$\widetilde{u}(S_{\lambda,v,A})(y)=\frac{e^{i\langle    y, v \rangle}-1}{\vert  y \vert^{t+\frac{\ell}{2}}}.$$ 
\end{lemma} 

\noindent
Notice that $\widetilde{u}$ ranges indeed in square-summable function classes: the condition $t>0$ ensures the integrability in a neighborhood of infinity and the condition $t<1$ ensures the integrability in a neighborhood of the origin.

\begin{proof}
Given any class in ${\rm H}^{1}(\simil, \chi_{t}\otimes \pi_{0})$, let $\beta\colon \simil\to \LLl$ be a cocycle representing it. We claim that upon replacing $\beta$ by an equivalent cocycle, we can assume that $\beta$ vanishes on the subgroup $\{ S_{\lambda,0,A}\}\simeq \R_{+}^{\ast}\times {\rm O}(\ell)$ of $\simil$. This is equivalent to the fact that $\beta$ satisfies Equation~(\ref{cond:conj}) above. To prove the claim, we must prove that the subgroup $\R_{+}^{\ast}\times {\rm O}(\ell)$ of $\simil$ fixes a point of $E$ in the affine action given by
\begin{equation}\label{eq:vanish_lambda}
\xi \longmapsto \lambda^{t}\pi_{0}(S_{\lambda , v, A})(\xi) + \beta(S_{\lambda,v,A}). 
\end{equation}
But if $\lambda <1$, the element $S_{\lambda, 0, \Id}$ acts by contraction on $E$ under the affine action above. Hence it has a unique fixed point $\xi_0$. Since the group $\R_{+}^{\ast}\times {\rm O}(\ell)$ commutes with $S_{\lambda , 0 , \Id}$, it must also fix $\xi_{0}$.   

From now on we thus assume that $\beta$ satisfies Equation~(\ref{cond:conj}) above. We choose a non-zero vector $b\in \R^{\ell}$. Let $g$ be the measurable function on $\R^{\ell}$ defined by:
$$g(y)=\frac{\beta(S_{1,b,\Id})(y)}{e^{i\langle y,b\rangle}-1}.$$
\noindent
Note that, because of the cocycle relation, $g$ does not depend on the choice of $b$:  indeed using the fact that $S_{1,b_{1},\Id}$ and $S_{1,b_{2},\Id}$ commute (for $b_{1}, b_{2}\in \R^{\ell}$) and the fact that $\beta$ is a cocycle for the representation $\chi_{t}\otimes \pi_{0}$, one has:
$$\begin{array}{rcl}
\beta(S_{1,b_{1}+b_{2},\Id}) & = & \beta(S_{1,b_{1},\Id})+\pi_{0}(S_{1,b_{1},\Id})(\beta(S_{1,b_{2},\Id}))\\
 & = & \beta(S_{1,b_{2},\Id})+\pi_{0}(S_{1,b_{2},\Id})(\beta(S_{1,b_{1},\Id})).\\
\end{array}$$
This can be rewritten as 
$$\beta(S_{1,b_{1},\Id})-\pi_{0}(S_{1,b_{2},\Id})(\beta(S_{1,b_{1},\Id}))=\beta(S_{1,b_{2},\Id})-\pi_{0}(S_{1,b_{1},\Id})(\beta(S_{1,b_{2},\Id})),$$
which is equivalent to
$$(1-e^{i\langle y, b_{2}\rangle})\beta(S_{1,b_{1},\Id})(y)=(1-e^{i\langle y, b_{1}\rangle})\beta(S_{1,b_{2},\Id})(y)$$
for almost every $y$. This indeed says that the function $g$ does not depend on the choice of $b\in \R^{\ell}-\{0\}$. 

We now consider the equation
\begin{equation}\label{ce}
\beta(S_{\lambda \mu,\lambda Ad+b,AB})=\beta(S_{\lambda,b,A})+\lambda^{t}\pi_{0}(S_{\lambda,b,A})(\beta(S_{\mu,d,B})),
\end{equation}
(which says that $\beta$ is a cocycle). Using the function $g$ it can be rewritten: 
$$\begin{array}{ccl}
\left( e^{i\langle y,\lambda Ad+b\rangle}-1\right)g(y) & = & \left( e^{i\langle y,b\rangle}-1\right) g(y)\\
 & & +\lambda^{t+\frac{\ell}{2}}e^{i\langle y,b\rangle}\left( e^{i\langle \lambda A^{-1}y,d\rangle}-1\right)g(\lambda A^{-1}y),\\
 \end{array}$$
which yields, after simplification:
\begin{equation}\label{equationtorduepourg_triv}
g(y)=\lambda^{t+\frac{\ell}{2}}g(\lambda A^{-1}y).  
\end{equation}
By taking $\lambda=1$ and letting $A$ vary in ${\rm O}(\ell)$, we see that $g$ is constant on the unit sphere of $\R^{\ell}$. The equation $g(y)=\lambda^{t+\frac{\ell}{2}}g(\lambda y)$ then implies that $$g(y)=\frac{a}{\vert y \vert^{t+\frac{\ell}{2}}}$$
for some constant $a$. In other words, $\beta$ is proportional to the cocycle $\widetilde{u}$. This proves that the group ${\rm H}^{1}(\simil, \chi_{t}\otimes \pi_{0})$ is generated by the cohomology class of $\widetilde{u}$. 

We still need to check that this class is non-zero. Denote by $g_{0}$ the function defined by $g_{0}(y)=\vert y \vert^{-(t+\frac{\ell}{2})}$. Observe that although $g_{0}$ is not square-summable one can still define the action of the operators $\chi_{t}\otimes \pi_{0}(S_{\lambda, v,A})$ on $g_{0}$. One has:
$$\chi_{t}\otimes \pi_{0}(S_{\lambda, v,A})(g_{0})(y)=\frac{e^{i\langle y , v \rangle}}{\vert y \vert^{t+\frac{l}{2}}},$$
and hence
$$\chi_{t}\otimes \pi_{0}(S_{\lambda, v,A})(g_{0})(y)-g_{0}(y)=\frac{e^{i\langle y , v \rangle}-1}{\vert y \vert^{t+\frac{l}{2}}}.$$
In other words, $\widetilde{u}$ is the formal coboundary of the function $g_{0}$. Now, if a function $f \in \LLl$ satisfies $\widetilde{u}(S_{\lambda,v,A})=\lambda^{t}\pi_{0}(S_{\lambda,v,A})f-f$, the function $f-g_{0}$ is invariant by the representation $\chi_{t}\otimes \pi_{0}$ of $\simil$. But any measurable function invariant by the action of the translation subgroup $\R^{\ell}\triangleleft \simil$ must be zero since $S_{1,v,\Id}$ acts by multiplication by $e^{i\langle y,v\rangle}$. Hence $f=g_{0}$. Since $g_{0}$ is not square-summable, we have a contradiction. Hence there is no function $f \in \LLl$ whose coboundary is equal to $\widetilde{u}$. This proves that the cohomology class of $\widetilde{u}$ is non-zero.  
\end{proof}
 
We now study the complexification of the orthogonal representation~$\pi$. 
 
\begin{prop}\label{gorbul}
The complexification $\pi_{\C}$ of $\pi$ is isomorphic to $\pi_{0}$; therefore the space ${\rm H}^{1}(\simil, \chi_{t}\otimes \pi_{\C})$ is $1$-dimensional.
\end{prop}

\begin{proof}
Recall from condition~\eqref{cond:no-inv} that $\pi$ has no $\R^{\ell}$-invariant vectors; thus, applying Mackey's Theorem (Theorem 14.1 in~\cite{mackey52}), the representation $\pi_{\C}$ can be constructed in the following way as an induced representation:

Choose a unit vector $e$ in $\mathbf{R}^{\ell}$ with stabilizer ${\rm O}(\ell-1)$. Choose a measurable section $s\colon \bS^{\ell-1} \to {\rm O}(\ell)$ of the corresponding projection ${\rm O}(\ell)\to  \bS^{\ell-1} \cong {\rm O}(\ell) / {\rm O}(\ell-1)$. This determines a cocycle $\gamma\colon {\rm O}(\ell)\times  \bS^{\ell-1}\to {\rm O}(\ell-1)$ by the rule $\gamma(A, v) = s(A v)^{-1} A s(v)$. We extend it to $v\in \R^{\ell}-\{0\}$ by $\gamma(A, v) =\gamma(A, v/|v|)$. Then there exists a unitary representation $\eta$ of ${\rm O}(\ell-1)$ on a Hilbert space $U$ such that $\pi_{\C}$ is, up to isomorphism, the representation on ${\rm L}^{2}(\R^{\ell},U)$ defined by
\begin{equation}\label{pic}
\left( \pi_{\C}(S_{\lambda,v,A})\cdot f\right)(y)=\lambda^{\frac{\ell}{2}}e^{i\langle y, v\rangle} \eta(\gamma(A^{-1},y)^{-1} ) f(\lambda A^{-1}y).
\end{equation}
Note that any decomposition of $U$ into a sum of two representations induces a decomposition of $\pi_{\C}$. We will prove the following claims: (a) $\pi_0$ appears as a sub-representation of $\pi_{\C}$ with multiplicity at most $1$; (b) any representation of $\simil$ of the form~\eqref{pic} (for some cocycle $\gamma$ and some unitary representation $\eta$) and with nontrivial first cohomology must contain a copy of $\pi_{0}$. These two points imply that $\pi_{\C}$ is isomorphic to $\pi_{0}$; the statement about the first cohomology group then follows from Lemma~\ref{dimco}. 

The cocycle $u$ (which originally ranges in the Hilbert space~$E$) is now thought of as a cocycle ranging in $E\otimes \C \simeq {\rm L}^{2}(\R^{\ell},U)$. Once again we fix $b\in \R^{\ell}-\{0\}$ and define a measurable function $g\colon \R^{\ell}\to U$ by
$$g(y)=\frac{u(S_{1,b,\Id})(y)}{e^{i\langle y, b\rangle}-1}.$$
As in Lemma~\ref{dimco}, it does not depend on the choice of $b$. We now rewrite equation~\eqref{ce} for the cocycle $u$ instead of the cocycle $\beta$ appearing in the proof of Lemma~\ref{dimco}. The equation becomes:
$$u(S_{\lambda \mu,\lambda Ad+b,AB})=u(S_{\lambda,b,A})+\lambda^{t}\pi_{\C}(S_{\lambda,b,A}) u(S_{\mu,d,B}).$$
Using condition~\eqref{cond:conj} on the cocycle $u$ and the function $g$ that we introduced, we can repeat the calculation made in the proof of Lemma~\ref{dimco}, taking into account the presence of the cocycle $\gamma$ and the unitary representation $\eta$. This leads to
\begin{equation}\label{equationtorduepourg}
g(y)=\lambda^{t+\frac{\ell}{2}}\eta(\gamma(A^{-1}, y)^{-1}) g(\lambda A^{-1}y).
\end{equation}

We now turn to the proof of claim (a). The representation $p\pi_{0}$ of $\simil$ is equivalent to the representation given by the right-hand side of~\eqref{pic} with $\eta$ trivial and $U$ of dimension $p$. So we assume for a moment that $\eta$ is trivial (up to replacing $\pi_{\C}$ by a sub-representation and up to taking the projection of $u$ to that sub-representation). In this case~\eqref{equationtorduepourg} reduces to the relation~\eqref{equationtorduepourg_triv} in the proof of Lemma~\eqref{dimco}. Thus we have again some vector $v\in U$ with $g(y)=\vert y \vert^{-t-\ell/2} v$, or in other words
$$u(S_{1,v,\Id})(y)=\frac{e^{i\langle y , v \rangle}-1}{\vert y \vert^{t+\frac{\ell}{2}}} v.$$
Since the complex vector space spanned by the family $u(S_{1,v,\Id})$ (for $v\in \R^{\ell}$) is dense in ${\rm L}^{2}(\R^{\ell},U)$ this implies that $U$ must be $1$-dimensional, thus proving claim (a).

We now prove claim (b). We consider~\eqref{equationtorduepourg} with $\lambda=1$; this states that $g$ yields a $\gamma$-equivariant map $\R^{\ell}-\{0\}\to U$ with respect to the cocycle $\gamma \colon {\rm O}(\ell)\times(\R^{\ell}-\{0\}) \to {\rm O}(\ell-1)$. The cocycle reduction lemma (see~\cite[p.~108]{Zimmer84}, applied to each sphere)  implies that 
$$\gamma(A,y)=\varphi^{-1}(Ay)\gamma'(A,y)\varphi(y)$$ where the equivalent cocycle $\gamma'$ ranges into the stabilizer in ${\rm O}(\ell-1)$ of a nonzero vector $\varepsilon\in U$. By considering the operator mapping a function $f\in {\rm L}^{2}(\R^{\ell},U)$ to the function $y\mapsto \eta(\varphi(y))f(y)$, we can assume that our representation is given by formula~\eqref{equationtorduepourg} with $\gamma'$ replacing $\gamma$. The subspace ${\rm L}^{2}(\R^{\ell},\C\cdot \varepsilon)$ of ${\rm L}^{2}(\R^{\ell},U)$ then gives a sub-representation isomorphic to $\pi_{0}$. This proves claim (b) and finishes the proof of the Proposition.\end{proof}

Having described the complexification $\pi_{\C}$ of $\pi$, we now go back to $\pi$ itself; the transition relies on the following fact.

\begin{lemma}
Let $I \colon  {\rm L}^{2}(\R^{\ell})\to {\rm L}^{2}(\R^{\ell})$ be an antilinear involution which commutes with the $\pi_0$-action of ${\rm Sim}(\R^{\ell})$. Then $I=e^{i\vartheta}I_{0}$ for some $\vartheta\in\R$, where $I_{0}$ is the involution defined by:
$$(I_{0}f)(y)=\overline{f(-y)}.$$
\end{lemma}

\begin{proof}
The map $I_{0}$ defined above is obviously an involution satisfying the conditions of the lemma. We want to prove that this is the only one, up to multiplication by a complex number of modulus $1$. If $I$ is another such involution, we can write $I=J\circ I_{0}$ where $J$ is complex-linear and commutes with the action of $\simil$. But any bounded complex linear operator which commutes with the action of $\R^{\ell}\triangleleft \simil$ is the multiplication by a function $h\in {\rm L}^{\infty}(\R^{\ell})$; indeed, the von Neumann algebra generated by the multipliers $e^{i\langle \cdot,v\rangle}$ where $v$ ranges over $\R^{\ell}$ is ${\rm L}^{\infty}(\R^{\ell})$, which is maximal abelian in the algebra of operators on ${\rm L}^{2}(\R^{\ell})$ (see e.g.~\cite[III.1]{Takesaki_I}). Hence:
$$(If)(y)=h(y)\overline{f(-y)}.$$\noindent The fact that $I$ commutes with the action of $\simil$ implies that $h$ is constant almost everywhere and the condition $I^{2}=\Id$ implies that $\vert h\vert^{2}=1$ almost everywhere. This completes the proof of the lemma.
\end{proof}

Now, for each real number $\vartheta$, the fixed point set of $e^{i\vartheta}I_{0}$ is isomorphic (as an orthogonal representation of $\simil$) to the fixed point set of $I_{0}$, via the map $f\mapsto e^{i\frac{\vartheta}{2}}f$. This implies that any two orthogonal representations of $\simil$ having $\pi_{0}$ as complexification are isomorphic. The cocycle $c$ defined above take its values in the fixed point set of $I_{0}$ and the fact that ${\rm H}^{1}(\simil,\chi_{t}\otimes \pi_{0})$ has dimension $1$ implies that ${\rm H}^{1}(\simil,\chi_{t}\otimes \pi)$ has dimension $1$ also. This proves that, up to isomorphism, there is a unique representation \mbox{$\alpha_{t} \colon  \simil \to {\rm Sim}(E)$} satisfying the condition of the beginning of the paragraph. 

\medskip
As explained earlier, this concludes the proof of Theorem~\ref{class}, up to the existence of the representations $\varrho_{t}$ that will be discussed in Section~\ref{princeseries}.

\section{The principal series and actions on hyperbolic spaces}\label{princeseries}

To construct the family of representations $\varrho_{t}$, we first need to recall a few well-known facts from representation theory (which can be found, for instance, in~\cite{johnsonwallach,knapp,sally1,sally2}). 

\subsection{Preliminaries from representation theory}\label{prelim-princeseries}

We continue to use the notation from paragraph~\ref{parabolic}. Let $o=(1,0,\ldots , 0)\in \HH$ and denote by $K$ its stabilizer in $\Isomn$. We denote by ${\rm Vol}$  the unique normalized $K^{\circ}$-invariant volume form on $\partial \HH$. Let ${\rm Jac}(g)(\cdot)$ be the Jacobian of $g\in \Isomn$ for the volume form ${\rm Vol}$; i.e. $(g^{\ast}{\rm Vol})_{b}={\rm Jac}(g)(b){\rm Vol}_{b}$ for $b\in \partial \HH$. If one thinks of ${\rm Vol}$ as a measure, $|{\rm Jac}(g)|=dg^{-1}_{\ast}{\rm Vol}/d{\rm Vol}$. Since we have $\vert {\rm Jac}(g)\vert =\vert {\rm Jac}(kg)\vert$ for $k\in K$, the function $\vert {\rm Jac}(g^{-1})(b)\vert$ descends to $\HH \times \partial \HH$; this is the {\it Poisson kernel}. One has (see~\cite{johnsonwallach}):

\begin{equation}\label{eqnjacobian}
\vert {\rm Jac}(g^{-1})(b)\vert =\left(\frac{B(o,b)}{B(g(o),b)}\right)^{n-1}.
\end{equation}

\noindent There is a slight abuse of notation in the formula above: on the left-hand side $b$ represents a point of the boundary of $\HH$ i.e. a $B$-isotropic line in $\R^{n+1}$. To compute the right-hand side, we can choose an isotropic vector representing the line $b$.

The {\it (spherical) principal series} is the family of representations $\pi_{s}$ of $\Isomn$ on the space ${\rm L}^{2}(\partial \HH)$, parametrized by a complex number $s$ and defined by: 
$$\pi_{s}(g)\cdot f= \vert {\rm Jac}(g^{-1})\vert^{\frac{1}{2}+s}f\circ g^{-1}.$$

\begin{rem}
Here ${\rm L}^{2}(\partial \HH)$ stands for the space of {\it complex-valued}, measurable, square-integrable function classes on $\partial \HH$ with respect to~$\mathrm{Vol}$, although we will soon restrict to {\it real-valued} functions, in which case we write ${\rm L}^{2}(\partial \HH, \R)$. Note that when $s$ is real, the operators $\pi_{s}(g)$ preserve the space ${\rm L}^{2}(\partial \HH, \R)$.
\end{rem}

\noindent If $f_{1},f_{2}\in {\rm L}^{2}(\partial \HH)$, we will write $(f_{1},f_{2})=\int_{\partial \HH}f_{1}\overline{f_{2}}$. A simple calculation shows that $(\pi_{s}(f_{1}),\pi_{-\overline{s}}(f_{2}))=(f_{1},f_{2})$; in other words the representations $\pi_{s}$ and $\pi_{-\overline{s}}$ are {\it dual}. It is known (see for instance~\cite{wallach}) that the representation $\pi_{s}$ is irreducible whenever \mbox{$s\neq \pm (\frac{1}{2}+\frac{k}{n-1})$} for all $k\in \mathbf{N}$.

Before going further, let us introduce a second model for the representations $\pi_{s}$. We will need to use a $KAN$ decomposition for the group $\Isomn$. We still consider the two isotropic vectors $\xi_{1}$ and $\xi_{2}$ considered in paragraph~\ref{parabolic}:
$$\xi_{1}=\frac{1}{\sqrt{2}}(1,1,0,\ldots),$$
$$\xi_{2}=\frac{1}{\sqrt{2}}(1,-1,0,\ldots),$$

\noindent and we still denote by $P$ the stabilizer of the line $\R \xi_{1}$ in $\Isomn$. Elements of $P$ are denoted by $g_{\lambda,v,A}$ where $\lambda \in \R_{+}^{\ast}$, $v\in \R^{n-1}$ and $A\in \mathrm{O}(n-1)$. Every element $g$ of $\Isomn$ can be written uniquely as a product:
$$g=k\cdot g_{\lambda,v,\Id},$$  

\noindent (where of course $k$, $\lambda$ and $v$ are functions of $g$ but we will not indicate it in the notation).

\medskip
Let $M$ be the centralizer of the $1$-parameter group $\{g_{\lambda,0,\Id}\}_{\lambda}$ in $K$. We can now define the second model for $\pi_{s}$.  Let $V_{s}$ be the space of measurable function classes $$F \colon  \Isomn \lra \C$$ 

\noindent such that: 
\begin{itemize}
\item $F(gmg_{\lambda,v,\Id})=F(g)\lambda^{-(\frac{1}{2}+s)(n-1)}$, for $g\in \Isomn$, $m\in M$, $\lambda\in \R_{+}^{\ast}$ and $v\in \R^{n-1}$,
\item the restriction of $F$ to $K$ is square-integrable. 
\end{itemize}
Denote by $dk$ the normalized Haar measure on $K$. The space $V_{s}$, endowed with the norm $\vert \cdot \vert$ defined by
$$\vert F\vert^{2}=\int_{K}\vert F(k)\vert^{2}dk,$$
\noindent is a Hilbert space. The group $\Isomn$ acts on it by translation; the element $g\in \Isomn$ acts by:
$$(g\cdot F)(h)= F (g^{-1}h).$$

\noindent The following proposition relates the previous representations.

\begin{prop}\label{identification}
The representation $\pi_{s}$ is isomorphic to the representation of $\Isomn$ on $V_{s}$. 
\end{prop}

\begin{proof}

Note that the map $k\mapsto [k(\xi_{1})]$ induces an identification between $K/M$ and $\partial \HH$. If $F\in V_{s}$, its restriction to $K$ descends to a square-integrable function on $K/M\simeq \partial \HH$. This defines a map from $V_{s}$ to $\LLB$ which is easily seen to be a surjective isometry. We will denote by
$$f\in \LLB \longmapsto F_{f}\in V_{s}$$
the inverse of this isometry. The point of the proof is to check that the action of $\Isomn$ on $V_{s}$ by  translations identifies with the representation $\pi_{s}$ under this isomorphism. This can be seen as follows. Let $f\in \LLB$; we will temporarily write $g\cdot f$ for the function which corresponds to $F_{f}(g^{-1}\cdot )$ under the isomorphism above. We want to check that $g\cdot f=\pi_{s}(g)(f)$. Let $b=[k\xi_{1}]$ be a point in the boundary of $\HH$. Let
$$g^{-1}k=k'g_{\lambda',v',\Id}$$
be the Iwasawa decomposition of $g^{-1}k$. A simple computation using Equation~(\ref{eqnjacobian}) shows that $\vert {\rm Jac}(g^{-1})(b)\vert^{\frac{1}{2}+s}=(\lambda')^{-(\frac{1}{2}+s)(n-1)}$. One then has:
$$\begin{array}{rcl}
(g\cdot f)(b) & = & F_{f}(g^{-1}k)\\
 & = & F_{f}(k')(\lambda')^{-(\frac{1}{2}+s)(n-1)}\\
 & = & \vert {\rm Jac}(g^{-1})(b)\vert^{\frac{1}{2}+s}f(g^{-1}b),\\
\end{array}$$
where, in the last step, we used the fact that $g^{-1}b=[g^{-1}k\xi_{1}]$ is equal to $[k'\xi_{1}]$. This proves that $g\cdot f =\pi_{s}(g)(f)$, as desired.
\end{proof}

From now on, we will assume that $s$ is a real number. We seek to construct a sesquilinear form on the space ${\rm L}^{2}(\partial \HH)$, invariant by the representation $\pi_{s}$. Since $\pi_{s}$ and $\pi_{-s}$ are dual, it is enough to construct an operator $L_{s}$ of $\LLB$ which intertwines the representations $\pi_{s}$ and $\pi_{-s}$, i.e. which satisfies:  
$$L_{s}\circ \pi_{s}(g)=\pi_{-s}(g)\circ L_{s}.$$

\noindent
It will be convenient to use both $\LLB$ and the model $V_{s}$ to describe the operator $L_{s}$. We first give a ``formal" definition which will then be justified by Proposition~\ref{prop:Ls}, as follows. If $F\in V_{s}$ and $F$ is continuous, we consider the function $L_{s} F$ defined by:
\begin{equation}\label{eq:Ls}
L_{s}F(x)=\frac{1}{\kappa(n,s)} \int_{\R^{n-1}}F(x g_{1,v,Id}\sigma)dv,
\end{equation}
\noindent where $dv$ stands for Lebesgue integration on $\R^{n-1}$, $\sigma$ is the involution defined in~(\ref{involution}) and $\kappa(n,s)>0$ is a normalization constant determined as follows. Viewing $L_s$ as an operator on $\LLB$, it is straightforward to check that $L_s$ maps the constant function $\boldsymbol{1}_{\partial\HH}$ to a (positive) constant function; choose $\kappa(n,s)$ such that this constant is one. One could, but probably shouldn't, compute
$$\kappa(n,s) = \frac{\Gamma\big((n-1)s\big)}{\Gamma\big((n-1)(s+\frac12)\big)}(2\pi)^{\frac{n-1}2} .$$
\medskip

The following proposition is classical; the reader will find its proof in (for instance)~\cite{vergne}, Proposition~1 page~91.

\begin{prop}
\label{prop:Ls}
Assume that $s>0$. Then if $F\in V_{s}$ is continuous, the integral defining $L_{s} F$ is convergent and defines a continuous function in $V_{-s}$.  Moreover, the operator \mbox{$L_{s} \colon V_{s}\to V_{-s}$} is continuous.\qed
\end{prop}

From now on, and in view of this result, we will always assume that $s>0$. Note that, by its very definition, the operator $L_{s}$ intertwines the actions of $\Isomn$ on $V_{s}$ and $V_{-s}$: an element $g\in \Isomn$ acts by left translation and the integral defining $L_{s}$ is a convolution on the right. Viewing $L_{s}$ as an operator from $\LLB$ to itself, it intertwines $\pi_{s}$ and $\pi_{-s}$ according to Proposition~\ref{identification}. We can thus define an invariant bilinear form $\langle \cdot , \cdot \rangle_{s}$ on $\LLB$ by
$$\langle f,g\rangle_{s}=\int_{\partial \HH}f\overline{L_{s}g}.$$
Observe that the normalization in~\eqref{eq:Ls} implies $\langle \boldsymbol{1}_{\partial\HH}, \boldsymbol{1}_{\partial\HH} \rangle_{s} =1$.

\smallskip

To continue further, we need to know whether the quadratic form $\langle \cdot , \cdot \rangle_{s}$ is of finite index and how its signature depends on $s$ (here by the {\it index} of a quadratic form we mean the maximal dimension of an isotropic subspace). This result is known and due to Johnson and Wallach~\cite{johnsonwallach} (for $n\ge 3$) and Sally~\cite{sally1,sally2} (for $n=2$); see also~\cite{takahashi}. Intertwining operators such as $L_{s}$ have also been studied for semisimple groups of higher rank, see~\cite{knapp,knapp0}.

\subsection{Invariant quadratic forms and hyperbolic spaces}\label{qfhs}
As the reader will soon see, the natural parameter $s$ for the spherical principal series is cumbersome to deal with geometric quantities. We shall freely use the following change of variable:

\begin{nota}
From now on, we write $t=(n-1)(s-\frac{1}{2})$.
\end{nota}

Let
$$\LLB=\bigoplus_{k=0}^{\infty}H^{k}$$
be the decomposition of $\LLB$ into $K$-irreducible components. Recall that, in the Klein model for $\HH$, where $\partial \HH \cong \bS^{n-1}$ is identified with the unit sphere in $\R^{n}$, the space $H^{k}$ is simply the space of (restrictions to the sphere of) harmonic homogeneous polynomials of degree $k$; recall that for $k\ge 2$, we have
\begin{equation}\label{eq:binom}
p_k := \dim H^k\ =\ \binom{n+k-1}{n-1} - \binom{n+k-3}{n-1}.  
\end{equation}
We also have $p_{0}:=\dim H^0=1$ and $p_{1}:= \dim H^1=n$.

\begin{rem} When $n=2$, it is more customary to decompose ${\rm L}^{2}(\partial \mathbf{H}^{2})$ into irreducible components for the group ${\rm SO}(2)$; in this case, each space $H^{k}$ ($k>0$) breaks into two invariant lines for ${\rm SO}(2)$. However, since we consider the action of the whole group $K={\rm O}(n)$, the spaces $H^{k}$ are irreducible, even in the case $n=2$.
\end{rem}

Since the action of $\pi_{s}(K)$ on $\LLB$ does not depend on $s$ and since $L_{s}$  commutes with this action, $L_{s}$ must preserve each of the subspaces $H^{k}$. Hence the restriction of $L_{s}$ to $H^{k}$ is the multiplication by a scalar which we will denote by $\lambda_{k}(s)$. In view of the normalization in~\eqref{eq:Ls}, we have $\lambda_{0}(s)=1$. It then follows that for $k\ge 1$ we have
\begin{equation}\label{valeurlambda}
\lambda_{k}(s)\ =\ \prod_{j=0}^{k-1}\frac{j+\frac{n-1}{2}-(n-1)s}{j+\frac{n-1}{2}+(n-1)s}\ =\ \prod_{j=0}^{k-1}\frac{j-t}{j+t+n-1},
\end{equation}
see~\cite{johnsonwallach} for $n\ge 3$ and~\cite{sally2} or~\cite{delzantpy} for $n=2$. From this formula one can easily compute the signature of $\langle \cdot , \cdot \rangle_{s}$. We write $s_{j}=\frac{j}{n-1}+\frac{1}{2}$ for the parameter $s$ corresponding to integer values $j\ge 0$ of $t$. When $s\in (0,s_{0})$, the bilinear form $\langle \cdot , \cdot \rangle_{s}$ is positive definite. When $s>s_{0}$ we have:
\begin{itemize}
\item if $s\in (s_{j},s_{j+1})$ with $j$ odd, $\langle \cdot , \cdot \rangle_{s}$ has index $\sum_{1\le k\le j;\, \text{$k$\,odd}}p_{k}$;
\item if $s\in (s_{j},s_{j+1})$ with $j$ even, $\langle \cdot , \cdot \rangle_{s}$ has index $\sum_{0\le k\le j;\, \text{$k$\,even}}p_{k}$.
\end{itemize}
Therefore, by~(\ref{eq:binom}), the possible values of the index of the form $\langle \cdot , \cdot \rangle_{s}$, when $s>s_{0}$, are exactly the numbers $\binom{n-1+j}{n-1}$ for $j\ge 0$.

\begin{exam}
When $n=2$, the index can be any positive integer.  
\end{exam}

\begin{exam}
When $n=3$, the possible indices are exactly the triangular numbers: $(j+1)(j+2)/2$ (with $j\ge 0$).
\end{exam}

The fact that the index is restricted by the formula $\binom{n-1+j}{n-1}$ will be discussed further in Section~\ref{higherrank}, see in particular Problem~\ref{prob:ranks} and Theorem~\ref{thm:gelf}.

\medskip

 For $s\in (s_{0},s_{1})$, the quadratic form has index~$1$ and the number $t$ ranges over the interval $(0,1)$; this is the parameter which appears in Theorem~\ref{class}. We will now use both the parameters $t$ and $s$ and always assume that $s\in (s_{0},s_{1})$. Observe that the coefficients $\lambda_{k}(s)$ are real (since $s$ is real), hence the quadratic form $\langle \cdot , \cdot \rangle_{s}$ is defined over $\mathbf{R}$. We can thus consider the subspace 
$${\rm L}^{2}(\partial \HH,\mathbf{R})\subset {\rm L}^{2}(\partial \HH)$$
of real-valued functions endowed with the form $\langle \cdot , \cdot \rangle_{s}$. Observe however that the space $\oplus_{k\ge 1}H^{k}$ endowed with the scalar product $-\langle \cdot , \cdot \rangle_{s}$ is {\it not} complete as shown by the next lemma. 

\begin{lemma}\label{tendverszero} The coefficients $\lambda_{k}(s)$ tend to $0$ as $k$ goes to infinity.
\end{lemma}
\begin{proof} The statement is equivalent to the divergence of the series $\sum_{j\ge 1}-\log(\frac{j-t}{j+t+n-1})$. But this series is equivalent to the harmonic series $\sum_{j\ge 1}\frac{1}{j}$, which diverges. 
\end{proof}

Hence the pair $({\rm L}^{2}(\partial \HH,\mathbf{R}),\langle \cdot , \cdot \rangle_{s})$ is {\it not} isomorphic to a pair $(\mathscr{H},B)$ as described in the introduction. However, if $V$ denotes the completion of the space $\oplus_{k\ge 1}H^{k}$ with respect to $-\langle \cdot , \cdot \rangle_{s}$, the group $\Isomn$ still acts on $H^{0}\oplus V$, preserving the extension of the form $\langle \cdot , \cdot \rangle_{s}$ to this space. Moreover, the linear action of $\Isomn$ on $H^{0}\oplus V$ is still irreducible. See for instance~\cite{delzantpy}, Section 2.1 and in particular Proposition~2 there, for more details. The arguments there are given for the case $n=2$ but work for all $n$. The need to complete the space ${\rm L}^{2}(\partial \HH,\R)$ will be illustrated in Proposition~\ref{ilfautcompleter}. 

From now on, we will denote by $\HHI_{t}$ the hyperbolic space associated to the pair $(H^{0}\oplus V,\langle \cdot , \cdot \rangle_{s})$ and by $\varrho_{t} \colon  \Isomn \to {\rm Isom}(\HHI_{t})$ the corresponding representation. The representation $\varrho_{t}$ is non-elementary and has no non-trivial, closed, totally geodesic invariant subspace.

Let $\ast$ be the point of $\HHI_{t}$ associated to the constant function~$1$. We record the following elementary observation:

\begin{lemma}\label{lemma:unique:k}
The point $\ast$ is the unique point of $\HHI_{t}$ fixed by $\mathrm{SO}(n)=K^\circ$ (a fortiori by $\mathrm{O}(n)=K$). Hence, the map $f_{t} \colon  \HH \to \HHI_{t}$ defined by
$$f_{t}(g\cdot o)=\pi_{s}(g)(\ast), \;\;\;\;\; (g\in \Isomn)$$
is the unique $\Isomn$-equivariant map from $\HH$ to $\HHI_{t}$.\end{lemma}

\begin{proof}
There is only one $K^\circ$-invariant line in $\LLB$ by transitivity of the $K^\circ$-action on $\partial \HH$. This is still true in the completion $H^{0}\oplus V$ since it is isomorphic as a $K$-module to $\LLB$.
\end{proof}

Note that the map
$$(\lambda,v,t)\longmapsto f_{t}(g_{\lambda,v,\Id}\cdot o)$$
is smooth in all its variables simultaneously. Note also that the map $f_{t}$ is {\it harmonic} (see for instance~\cite{nishi} for the definition): if the tension field $\tau (f_{t})$ of $f_{t}$ was non-zero, it would be non-zero at every point and the geodesic tangent to the vector $\tau (f_{t})(o)$ at $o$ would be (pointwise) $K$-invariant, contradicting the fact that $K$ has a unique fixed point in $\HHI_{t}$. 

\bigskip
Let $g_{u}=g_{e^{u},0,\Id}$. We want to evaluate the distance between the point $$\pi_{s}(g_{u})(\ast)=f_{t}(g_{u}\cdot o)$$
and $\ast=f_{t}(o)$ in $\HHI_{t}$. This distance is given by the formula:
\begin{equation}\label{eq:distance}
\cosh(d(\pi_{s}(g_{u})(\ast),\ast))=\int_{\partial \HH}\jac(g_{u}^{-1})^{\frac{1}{2}+s}db
\end{equation}
where $db$ stands for the normalized $K$-invariant volume form on $\partial \HH$. We will denote by $I_{u}$ the integral above. The formula~(\ref{eqnjacobian}) for the Jacobian yields
\begin{equation}\label{eq:jaco}
\jac(g_{u}^{-1}) (b) = \big(\cosh(u)-b_{1}\sinh(u) \big)^{-(n-1)} \kern10mm (b\in \partial \HH)
\end{equation}
where $b_{1}$ stands for the first coordinate of $b\in \bS^{n-1}\subset \R^{n}$ under the identification of $\partial \HH$ with the unit sphere $\bS^{n-1}$ via the Klein model. Recalling that $t=(n-1)(s-\frac{1}{2})$, we obtain:
\begin{equation}\label{eq:Iu}
I_{u}:=\cosh(d(\pi_{s}(g_{u})(\ast),\ast))=\int_{\bS^{n-1}}(\cosh(u)-b_{1}\sinh(u))^{-(n-1+t)} db.
\end{equation}

Let us mention that in~\cite{delzantpy}, where the representations $\varrho_{t}$ were studied in the case $n=2$, the following formula was established: 
$${\rm cosh}(d(\varrho_{t}(a_{u})(\ast),\ast)=\frac{1}{2\pi}\int_{0}^{2\pi}\frac{d\vartheta}{\vert {\rm sinh}(u)e^{i\vartheta}+{\rm cosh}(u)\vert^{2(t+1)}}$$

\noindent where $a_{u}$ is the following element of ${\rm SU}(1,1)\subset {\rm Isom}(\mathbf{H}^{2})$

$$\left( \begin{array}{cc}
{\rm cosh}(u) & {\rm sinh}(u)\\
{\rm sinh}(u) & {\rm cosh}(u)\\
\end{array}\right).$$
This is coherent with Equation~(\ref{eq:Iu}): indeed a straightforward computation shows that the integral above equals:
$$\frac{1}{2\pi}\int_{0}^{2\pi}\frac{d\vartheta}{\vert {\rm cosh}(2u)+{\rm sinh}(2u)\cos (\vartheta)\vert^{1+t}}.$$ 
Remembering that $a_{u}$ has translation length $2\vert u\vert$ in $\mathbf{H}^{2}$, one recovers a formula of the same form as~(\ref{eq:Iu}).

\subsection{\texorpdfstring{Geometric properties of the representations $\varrho_t$}{Geometric properties of the representations}}

To complete the proof of Theorem~\ref{class} and to prove Theorem~\ref{curv}, we need to obtain estimates on the quantity $I_{u}$. We first deal with its behavior when $u$ goes to zero. We will prove the following proposition (already established in~\cite{delzantpy} when $n=2$).

\begin{prop}
We have, as $u$ goes to zero:
$$d(\pi_{s}(g_{u})(\ast),\ast)=\sqrt{\frac{t(t+n-1)}{n}} \cdot \vert u\vert+O(u^{2}).$$
\end{prop}

We remark that it is not granted a priori that the displacement of~$\ast$ under $g_u$ should have an analytic expression in $u$ in infinite-dimensional spaces such as~$\HHI$; indeed this fails completely for actions on Hilbert spaces.

\begin{proof}
The expression $\left(\cosh(u)-b_{1}\sinh(u)\right)^{-(n-1+t)}$ admits a Taylor expansion at $u=0$, which to order~$3$ is
$$1+(t+n-1)b_{1}u-\frac{(t+n-1)}{2} u^{2} +\frac{(t+n-1)(t+n)}{2}b_{1}^{2}u^{2}+O(u^{3}).$$
Thus, after integrating,
\begin{equation}\label{eq:taylor}
\cosh(d(\pi_{s}(g_{u})(\ast),\ast)=1-\frac{(t+n-1)}{2} u^{2} +\frac{(t+n-1)(t+n)}{2}u^{2}c_{n}+O(u^{3}),
\end{equation}
where $c_{n}$ is the average of the function $b_{1}^{2}$ on the sphere, which is $c_n=1/n$ by an easy symmetry argument. On the other hand, we can write $d(\pi_{s}(g_{u})(\ast),\ast)=\lambda \vert u\vert +O(u^{2})$ for some constant $\lambda$, as follows e.g.\ from the existence of the expansion~(\ref{eq:taylor}) via the implicit function theorem. This implies that $$\cosh(d(\pi_{s}(g_{u})(\ast),\ast)=1+\frac{1}{2} \lambda^{2}u^{2}+O(u^{3}).$$
 By identification in the two expansions for $\cosh(d(\pi_{s}(g_{u})(\ast),\ast))$ we get $\lambda^{2}=t(t+n-1)/n$. This leads to the expression announced in the statement of the proposition.
\end{proof}

Since the two metrics $g_{hyp}$ and $f_{t}^{\ast}g_{hyp}$ on $\HH$ are proportional, the previous proposition implies that $f_{t}^{\ast}g_{hyp}=\frac{t(t+n-1)}{n}g_{hyp}$. In particular, the curvature of $\HH$ endowed with the metric $f_{t}^{\ast}g_{hyp}$ is equal to $-\frac{n}{t(t+n-1)}$.

\begin{prop}\label{asymptotic}
For $n$ and $t$ fixed, there is a constant $\kappa >0$ such that
$$\kappa \, e^{tu}\ \leq\ I_{u} \ \leq\ e^{tu}$$
for all $u\ge 0$. 
\end{prop}

\begin{proof}
We work with the expression~(\ref{eq:Iu}) for $I_u$ and begin with the upper bound. Since the term
\begin{equation}\label{eq:e+sinh}
\cosh(u)-b_1 \sinh(u) = e^{-u} + (1- b_1) \sinh(u)
\end{equation}
is~$\geq e^{-u}$, it suffices to prove that
$$\int_{\bS^{n-1}}(\cosh(u)-b_{1}\sinh(u))^{-(n-1)} db$$
remains bounded by one. But the above expression is exactly one since it is the integral of the Jacobian of~$g_{u}^{-1}$. For the lower bound, we shall only integrate~(\ref{eq:Iu}) over $1-e^{-2u}<b_1<1$. In that region, the identity~(\ref{eq:e+sinh}) yields
$$\cosh(u)- b_1 \sinh(u) \leq \frac32 e^{-u}$$
and thus the integrand is at least a constant times $e^{u(n-1+t)}$. An elementary computation shows that the cap of $\bS^{n-1}$ defined by $b_1>1-e^{-2u}$ has volume larger than $e^{-u(n-1)}$ times a constant depending only on $n$. The lower bound follows.
\end{proof}

Note that Proposition~\ref{asymptotic} immediately gives us the end of the proof of Theorem~\ref{class}. Since $\varrho_{t}$ preserves the type (Proposition~\ref{type}) and since all hyperbolic elements of $\Isomn$ are conjugate to an element of the form $\{g_{\lambda, 0, A}\}$, the functions $\ell_{\varrho_{t}}$ and $\ell_{\HH}$ are proportional. According to Proposition~\ref{asymptotic} the proportionality constant is equal to $t$. We now collect all the previous informations to obtain a proof of Theorem~\ref{curv}.

\begin{proof}[Proof of Theorem~\ref{curv}]
We have already seen that there is an $\Isomn$-equivariant map $f_{t} \colon  \HH \to \HHI_{t}$ which satisfies $f_{t}^{\ast}g_{hyp}=\frac{t(t+n-1)}{n}g_{hyp}$ (and is unique). We have also seen that $f_{t}$ is harmonic. Since the metric $f_{t}^{\ast}g_{hyp}$ is proportional to the hyperbolic metric on $\HH$, the harmonicity of $f_{t}$ is equivalent to the fact that the image of $f_{t}$ is a minimal sub-manifold, thus justifying the claim from Theorem~\ref{curv} (2). 

\smallskip
We now have to establish the first and third point in Theorem~\ref{curv}. Let $x$ and $y$ in $\HH$ and denote by $u$ the distance from $x$ to $y$. There exists and element $h\in \Isomn$ such that $x=h(o)$ and $y=hg_{u}(o)$ where $g_{u}=g_{e^{u},0,\Id}$ as in the previous paragraph. Hence $d:=d_{\HHI_{t}}(f_{t}(x),f_{t}(y))=d_{\HHI_{t}}(f_{t}(g_{u}(o)),f_{t}(o))=\cosh^{-1}(I_{u})$ where $I_{u}$ is the integral from the previous paragraph. Since $\frac{1}{2}e^{d}\le I_{u}\le e^{d}$, we deduce from Proposition~\ref{asymptotic} that
$$\log(\kappa)+td_{\HH}(x,y)\le d_{\HHI_{t}}(f_{t}(x),f_{t}(y))\le td_{\HH}(x,y)+\log (2).$$
This shows the first point.

\smallskip
For the last point, we recall that quasi-isometric maps between geodesic Gromov-hyperbolic spaces extend continuously to the boundary. This concludes the proof since the estimates of point~(1) are stronger than quasi-isometric.
\end{proof}

Knowing the translation length of a hyperbolic element $g$ of $\Isomn$ in the representation $\varrho_{t}$, we can now prove that the fixed points of $\varrho_{t}(g)$ in $\partial \HHI$ do not lie in ${\rm L}^{2}(\partial \HH,\R)$.

\begin{prop}\label{ilfautcompleter} Let $g\in \Isomn$ be a non-trivial hyperbolic element of translation length $u\ge 0$. Let $f\in H^{0}\oplus V$ be an isotropic vector representing the attracting fixed point of $\varrho_{t}(g)$ in $\partial \HHI_{t}$. Then $f$ does not lie in the dense subspace ${\rm L}^{2}(\partial \HH,\R)\subset H^{0}\oplus V$. 
\end{prop}
\begin{proof} We can assume that $g=g_{u}$ for some $u$. Note that if $\xi$ is an isotropic vector representing the attracting fixed point for $\varrho_{t}(g_{u})$ one has $\varrho_{t}(g_{u})(\xi)=e^{tu}\xi$. Assume that there exists a non-zero function $f\in {\rm L}^{2}(\partial \HH,\R)$ such that:
$$e^{tu}f(b)={\rm Jac}(g_{u}^{-1})^{s+\frac{1}{2}}(b)f(g_{u}^{-1}(b)).$$
Up to replacing $f$ by $|f|$, we can assume that $f\ge0$. We rewrite the above equation in the following way:
$$e^{tu}f(b)={\rm Jac}(g_{u}^{-1})^{\frac{t}{n-1}}(b) {\rm Jac}(g_{u}^{-1})(b)f(g_{u}^{-1}(b)).$$
Using formula~(\ref{eq:jaco}) for the Jacobian, one sees that the maximum value of $\jac(g_{u}^{-1})$ on $\partial \HH$ is equal to $e^{(n-1)u}$, we thus obtain:
$$e^{tu}f(b)={\rm Jac}(g_{u}^{-1})^{s+\frac{1}{2}}(b)f(g_{u}^{-1}(b))\le e^{tu}\jac(g_{u}^{-1})(b)f(g_{u}^{-1}(b)).$$
Being in ${\rm L}^{2}(\partial \HH,\R)$, the function $f$ is in particular integrable; we can thus integrate each side of the above equation. Since the two integrals are equal we must have:
$${\rm Jac}(g_{u}^{-1})^{s+\frac{1}{2}}(b)f(g_{u}^{-1}(b))= e^{tu}\jac(g_{u}^{-1})(b)f(g_{u}^{-1}(b))$$
almost everywhere. If $f$ is non-zero, this implies that the following identity holds on a set of positive measure:
$${\rm Jac}(g_{u}^{-1})^{\frac{t}{n-1}}(b)=e^{tu}.$$
This is a contradiction. 
\end{proof}

\begin{rem}\label{indice2} In this text, we chose to deal with the full isometry group $\Isomn$ of the hyperbolic space $\HH$. Note however that the actions $\varrho_{t}$ are still non-elementary when restricted to the identity component $\Isomn^{\circ}$ of $\Isomn$. This follows for instance from the fact that an action of a group $G$ on $\HHI$ is elementary if and only if it preserves a finite set in $\HHI \cup \partial \HHI$.  Also, it is easy to see that our proof that the $\varrho_{t}$ are the only irreducible actions of $\Isomn$ on $\HHI$ still applies to the group $\Isomn^{\circ}$.
\end{rem}


\section{Convex hulls and exotic spaces}\label{sec:Ct}
This section is devoted to the study of the spaces $C_{t}$. In Section~\ref{sec:Ct:def}, we establish their surprising properties described in Theorem~\ref{convexhull} and Proposition~\ref{propertiesconvexhull}. Section~\ref{sec:GH} introduces and studies the \emph{strong topology}, which is a refined Gromov--Hausdorff topology on the class of all pointed metric $G$-spaces (where $G$ is a fixed group). Section~\ref{preuvecontinuity} contains the proof of Theorem~\ref{continuous}. Finally, Section~\ref{renor} contains a few remarks concerning the behavior of the spaces $C_{t}$ as $t$ goes to $0$.

\subsection{\texorpdfstring{The spaces $C_t$}{The spaces Ct}}\label{sec:Ct:def} 
Let $C_t$ be the closed convex hull of $f_{t}(\HH)$ inside $\HHI$. We have: 

\begin{lemma}\label{lemma:C_exists} The  space $C_{t}$ is minimal, i.e. it admits no nontrivial $\Isomn$-invariant closed convex subset. Moreover, this is the unique minimal $\Isomn$-invariant closed convex subset of $\HHI$. 
\end{lemma}

Actually, it is easy to deduce from the results in~\cite{monodJAMS} that any group $G$ acting isometrically on a \cat{-1} space $X$ and having no fixed point in $X\cup \partial X$ has a unique minimal closed $G$-invariant convex subset in $X$. However, in our case the fact that the compact subgroup $K\subset \Isomn$ has a unique fixed point in $\HHI$ makes the proof immediate. 

\begin{proof}[Proof of Lemma~\ref{lemma:C_exists}] It is enough to prove that if $C$ is a closed $\Isomn$-invariant convex subset of $\HHI$, then $C$ contains $C_{t}$. But if $C$ is such a set, the compact group $K$ must fix a point in $C$. Since $\ast$ is the unique $K$-fixed point in $\HHI$, we have $\ast\in C$. Hence the $\Isomn$-orbit of $\ast$ and its closed convex hull are contained in $C$, i.e. $C_{t}\subset C$. 
\end{proof}

Theorem~\ref{curv}(\ref{curv:bd}) yields a continuous $\Isomn$-equivariant map $\partial\HH\to\partial\HHI$. Its image is a non-empty  $\Isomn$-invariant subset $f_t(\partial\HH)$ which is compact for the cone topology. We now consider the Klein model: that is, the space $\HHI$ is identified with the unit ball $B$ in a Hilbert space and hyperbolic geodesics correspond simply to chords in that ball.  Moreover, the topology on $\overline{\HHI}$ corresponds to the norm topology on the closed ball $\overline B$.

\smallskip
Denote by $C'\se\HHI$ the closed convex hull of $f_t(\partial\HH)$, which is by definition the intersection of all closed convex subsets of $\HHI$ whose boundary at infinity contains $f_t(\partial\HH)$. The above discussion shows that $C'$ is identified with the intersection of $B$ and the closed convex hull (in the affine sense) of $f_t(\partial\HH)$ in $\overline B$. It also makes it apparent that $\partial C' = f_t(\partial\HH)$. The Mazur compactness theorem~\cite{mazur} (an English reference is e.g.~\cite[2.8.15]{meg}) shows that the closed convex hull of $f_t(\partial\HH)$ in $\overline B$ is norm-compact. Therefore, it follows that $C'$ is locally compact and closed in $B$. In conclusion, $C'$ is a closed subspace of $\HHI$ and is proper as a metric space.

\smallskip

\begin{lemma}\label{lemma:all_C}
We have $C'=C_t$. In particular, $C_t$ is a proper metric space in view of the above application of Mazur's compactness theorem.
\end{lemma}

\begin{proof}
By definition, $C_{t}$ is the intersection of all closed convex subsets of $\HHI$ that contain $f_t(\HH)$. Since the boundary of any such set contains $f_t(\partial\HH)$, we have $C'\se C_{t}$. By minimality of $C_{t}$ (Lemma~\ref{lemma:C_exists}) we have $C'=C_{t}$.
\end{proof}

\begin{prop}\label{prop:coco}
The $\Isomn$-action on $C_t$ is cocompact.
\end{prop}

\begin{proof}
Suppose for a contradiction that $C_t$ contains a sequence $x_j$ such that the distance $d(x_j, \Isomn \ast )$ to the orbit of $\ast =f_t(o)$ tends to infinity. Upon replacing $x_j$ with an $\Isomn$-translate, we can suppose that $d(x_j, \Isomn \ast ) = d(x_j, \ast )$. Moreover, we can assume that $x_j$ converges to some point $\xi\in \partial C_t = f_t(\partial\HH)$. Choose a sequence $y_i$ in $\HH$ such that $f_t(y_i)$ converges also to $\xi$. Thus the Gromov product
$$L_{i,j}\ =\ \frac12 \Big(  d(f_t(y_i), \ast )  + d(x_j, \ast ) - d(f_t(y_i), x_j) \Big)$$
tends to infinity as $i,j\to \infty$. Let $z_{i,j}$ be the point of the geodesic $[\ast , x_j]$ at distance $L_{i,j}$ of $\ast $. Since $d(x_j, \Isomn \ast ) = d(x_j, \ast )$, we have $d(z_{i,j}, \Isomn \ast ) = L_{i,j}\to\infty$. By Gromov-hyperbolicity, there is a constant bounding the distance from $z_{i,j}$ to the geodesic $[\ast , f_t(y_i)]$. Theorem~\ref{curv}(\ref{curv:qi}) shows that the image under $f_t$ of the geodesic $[o, y_i]$ is a quasi-geodesic controlled by constants depending only on $t$. Thus, by Gromov-hyperbolicity, a similar constant bounds the distance from $[\ast , f_t(y_i)]$ to $f_t(\HH)$. Putting everything together, $z_{i,j}$ remains at bounded distance of $f_t(\HH)$, which is none other than the orbit of $\ast $; this is a contradiction.
\end{proof}

We now endow the group $\Isom(C_t)$ of isometries of $C_t$ with the compact-open topology. Since the space $C_t$ is proper, this turns it into a locally compact group. 

\begin{prop}\label{prop:isom:ct}
The representation $\varrho_t$ induces an isomorphism $\Isomn\cong\Isom(C_t)$.
\end{prop}

\begin{proof}
For simplicity, let $G=\Isom(C_t)$ and $L=\Isomn$. Since $\varrho_t$ induces a proper homomorphism, we consider $L$ as a closed subgroup of $G$. Moreover, $L$ is cocompact in $G$ because it acts cocompactly on $C_t$. Since $L$ is almost connected ($L^\circ$ has index two in $L$), the connected component $G^\circ$ is cocompact in $G$. Therefore, Yamabe's theorem~\cite[4.6]{montgomeryzippin} implies that $G$ is a Lie group after possibly factoring out a compact normal subgroup $K\lhd G$. However, $G$ has no non-trivial compact normal subgroup because the fixed-point set $C_t^K$ is a non-empty closed convex $G$-invariant set, which implies $C_t^K=C_t$ by minimality of $C_t$ (even amongst $L$-invariant such sets).

Thus, we have a connected Lie group $G^\circ$ of finite index in $G$ (and containing $L^\circ$). The amenable radical of $G^\circ$ is trivial (see e.g.~\cite[1.10]{capracemonod}), so that $G^\circ$ is a semi-simple Lie group without compact factors. Since $L^\circ$ is closed cocompact in $G^\circ$, this forces $G^\circ = L^\circ$ (e.g.\ using the Karpelevich--Mostow theorem). At this point, we have realized $G$ as an extension
$$1 \lra L^\circ \lra G  \lra F \lra 1$$
with $F$ finite. We know from Lie theory that the group of outer automorphisms $\Out(L^\circ)$ is of order two, with the non-trivial element being implemented by conjugation by orientation-reversing elements of $L$. The above group extension induces a representation $G\to F\to\Out(L^\circ)$ whose kernel is $L^\circ. Z_G(L^\circ)$, where $Z_G(L^\circ)$ is the centralizer of $L^\circ$. The centralizer has to be trivial by minimality of $C_t$, see e.g.~\cite[1.10]{capracemonod}.  Therefore, $L^\circ$ is of order two in $G$ and thus $G=L$, finishing the proof.
\end{proof}

\begin{cor}
If $t, t'\in (0, 1]$ are distinct, then the spaces $C_t$ and $C_{t'}$ are not isometric, even upon rescaling the metric.
\end{cor}

\begin{proof}
Suppose that there is some scalar $\mu>0$ and a bijection $h\colon C_t\to C_{t'}$ such that $d(h(x), h(y)) = \mu d(x,y)$ for all $x,y\in C_t$. The conjugation of $\varrho_{t'}$ by $h$ yields a continuous homomorphism $\varrho_{t'}^h\colon \Isomn\to \Isom(C_t)$. In view of Proposition~\ref{prop:isom:ct}, the maps $\varrho_{t'}^h$ and $\varrho_t$ are isomorphisms onto $\Isom(C_t)$ and differ by an automorphism of $L=\Isomn$. The description of the outer automorphisms of $L^{\circ}$ given in the proof of Proposition~\ref{prop:isom:ct} implies that the group $L=\Isomn$ has trivial outer automorphism group upon considering the natural map from $\Out(L)$ to $\Out(L^{\circ})$. In conclusion, the two actions on $C_t$ are conjugated by an element of $\Isomn$ and thus $h$ is $\Isomn$-equivariant after conjugation.

We can now apply point~(\ref{class:exist}) of Theorem~\ref{class} and deduce $\mu t = t'$. Moreover, since the subset $f_t(\HH)$ is intrinsically characterized in terms of (the conjugacy class of) the action $\varrho_t$ on $C_t$ by Lemma~\ref{lemma:unique:k}, the curvature given by Theorem~\ref{curv}(\ref{curv:curv}) must rescale according to $\mu^{-2}$. Thus we have
$$\frac{-n}{t'(t'+n-1)}\ =\ \frac{-n}{\mu^{2} t(t+n-1)}$$
which together with $\mu t = t'$ implies $t=t'$, as desired.
\end{proof}

At this point the proofs of Theorem~\ref{convexhull} and Proposition~\ref{propertiesconvexhull} are complete. 

\bigskip

The local compactness of the exotic space $C_t$ established in Proposition~\ref{prop:coco} could not have happened if the ambient space were a Hilbert space rather than $\HHI$, as shown by the following fact.

\begin{lemma}
Let $G$ be a simple Lie group (or a simple algebraic group over a local field). If $G$ acts continuously by isometries on a Hilbert space $V$, then either it fixes a point or it does not preserve any non-empty closed locally compact convex subspace of $V$.
\end{lemma}

\begin{proof}
We can assume that the linear part of the action has no $G$-invariant vectors.  Indeed, since $G$ has no non-zero additive characters, the $G$-invariant part splits off a summand preserved by the isometric action. Suppose for a contradiction that $G$ preserves a non-empty closed locally compact convex subset $C\se V$ but has no global fixed point. In particular, $G$ is non-compact and hence contains an element $g$ generating an unbounded subgroup of $G$ (e.g.\ a split torus element). Since $C$ is a proper \cat0 space, $g$ fixes a point either in $C$ or in its boundary. In the former case, all of $G$ will have a fixed point by a Mautner type argument, which is absurd (we already recalled earlier that unbounded $G$-actions are proper). In the latter case, the linear part of the action of $g$ fixes a vector. This implies that the linear $G$-representation has a fixed vector by Moore's theorem (\cite{moore}; see also~\cite[5.5]{howe-moore}), contrary to our preliminary reduction step.
\end{proof}

\subsection{\texorpdfstring{The strong topology on pointed metric $G$-spaces}{The strong topology on pointed metric G-spaces}}\label{sec:GH}
For an introduction to the \emph{pointed Gromov--Hausdorff topology}, see~\cite{bestvina,pauling} as well as~\cite[I.5]{bh} and~\cite[chap.~3]{gromov0}. It builds upon the Hausdorff distance defined in~\cite{hausdorff} (Kap.~VIII \S\,6, pp.~293 sqq). We shall now introduce a natural ``strong'' topology for actions on metric spaces, which is an equivariant version of the pointed Gromov--Hausdorff topology. Regarding terminology, the reader should be warned that the ``equivariant Gromov--Hausdorff topology'' introduced by Paulin~\cite{pauling} is very different from the ``strong'' topology used below; for instance, Paulin's topology does not imply convergence of the underlying spaces.

\medskip 
Recall that the \emph{coradius} of a subset $A$ of a metric space $X$ is the (possibly infinite) number
$$\corad_X(A)\ =\ \sup_{x\in X}\inf_{a\in A} d(x,a).$$
From now on and all along Section~\ref{sec:GH}, we will assume that all metric spaces are {\it geodesic}. If $X$ is a metric space and $x$ a point in $X$ we will denote by $B_{X}(x,r)$ the closed ball of radius $r$ around $x$ in $X$.

\bigskip
Let $G$ be a topological group. A \emph{pointed metric $G$-space} is a triple $(X, *, \pi)$ where $X$ is a metric space, $*\in X$ and $\pi$ is a continuous $G$-action by isometries on $X$; we denote also by $\pi\colon G\to \Isom(X)$ the corresponding homomorphism and write often $gx$ for $\pi(g)(x)$. We define a topology on the class of all pointed metric $G$-spaces as follows.

\begin{defi}\label{def:topol}
Let $(X, *, \pi)$ be a pointed metric $G$-space. Given $R>0$, $\epsilon>0$ and a compact subset $M\se G$ such that $1\in M$, let $U_{R, \epsilon, M}(X, *, \pi)$ be the class of all pointed metric $G$-spaces $(X', *', \pi')$ such that there is an \emph{$(M, \epsilon)$-map} $f$ relative to the balls $B_{X'}(*', R)$ and $B_{X}(*, R)$. This shall mean, by definition, that $f$ is a map $f:X'\to X$ such that

\smallskip
\begin{enumerate}
\item $f(*') = *$,
\item the Hausdorff distance between $f(B_{X'}(*', R))$ and $B_{X}(*, R)$ is $<\epsilon$,
\item $\Big| d(f(x), f(y)) - d(x,y) \Big| <\epsilon$ for all $x,y\in \pi'(M) \big(B_{X'}(*', R)\big)$,
\item $d(g f(x), f(g x)) < \epsilon$ for all $x\in B_{X'}(*', R)$ and all $g\in M$.\label{presqueequivariance}
\end{enumerate}

\smallskip
It is clear that these $U_{R, \epsilon, M}(X, *, \pi)$ form a neighborhood base. We shall refer to this as the topology of \emph{strong convergence} of pointed metric $G$-spaces.
\end{defi}

 This topology is very close to the one considered by Fukaya~\cite{fukaya}. Notice that when $G$ is the trivial group, we essentially recover the pointed Gromov--Hausdorff convergence of metric spaces~\cite{gromov0}. Observe also that, if $G$ is $\sigma$-compact (e.g.\ locally compact and connected or second countable), then every pointed metric $G$-space has a countable neighborhood base. 

\smallskip

Before turning to the proof of the continuity of the map $t\in (0,1]\mapsto (C_{t},\ast,\varrho_{t})$ for the topology just defined, we prove a preliminary fact concerning isometries of \cat{-1} spaces. We recall that a hyperbolic isometry $g$ of a complete \cat{-1} space $X$ admits a unique axis $L_g\se X$. We shall denote by $\xi^+_g, \xi^-_g\in\partial X$ its forward, respectively backward endpoints and by $\ell_X(g)$ the translation length of $g$. We recall, but shall not use, that for any $x\in X$ one has $g^n x\to\xi^\pm_g$ as $n\to\pm\infty$.  

\begin{prop}\label{prop:xi_cont}
Let $X$ be a complete \cat{-1} space and endow the set $H$ of its hyperbolic isometries with the topology of pointwise convergence. Then the map
$$\begin{array}{rcl}
H & \lra & \partial X\\
 g & \longmapsto & \xi^+_g\\
 \end{array}$$
is continuous on each level-set of $\ell_X\colon H\to \R_+$.
\end{prop}

\noindent
A more refined analysis reveals in fact that continuity holds whenever $\ell_X$ remains bounded away from zero, but we will not need this. 

This proposition will be used later (see Lemma~\ref{continuityboundary}) in the following way: to prove that the spaces $(C_{t},\ast)$ depend continuously on $t$ we will first prove that the limit sets $f_{t}(\partial \HH)\subset \partial \HHI$ depend continuously on $t$. But $f_{t}(\partial \HH)$ is nothing else than the union of the attractive fixed points in $\partial \HHI$ of the hyperbolic isometries in 
$$\varrho_{t}(\Isomn).$$ 
The above proposition will thus be useful to establish the continuity of the map $t \mapsto f_{t}(\partial \HH)$. 

\medskip

We now turn to the proof of Proposition~\ref{prop:xi_cont}. We shall use the following fact from \cat{-1} geometry.

\begin{lemma}\label{lem:comp}
Let $h$ be a hyperbolic isometry of a complete \cat{-1} space $X$. There is a continuous, increasing function $c\colon[\ell_X(h), \infty) \to[0, \infty)$ depending only on $\ell_X(h)$ with $c(\ell_X(h))=0$ and such that for any $x\in X$ we have $d(x, L_h) \leq c(d(h x, x))$.
\end{lemma}

\begin{proof}[Proof of the lemma]
Fix $r>0$. Elementary geometry of the hyperbolic plane $\mathbf{H}^2$ shows that there is a continuous, increasing function $c\colon[r, \infty) \to[0, \infty)$ with $c(r)=0$ and satisfying the following property. If $\bar x,\bar y,\bar z,\bar t$ are four points in $\mathbf{H}^2$ with $d(\bar z,\bar t)=r$ and with (hyperbolic) angles $\angle_{\bar z}(\bar x,\bar t)$ and $\angle_{\bar t}(\bar y,\bar z)$ both at least~$\pi/2$, then $d(\bar x,\bar z)\leq c(d(\bar x,\bar y))$ (and $d(\bar x,\bar y)\geq r$). In fact, we will use it only for the special case where $d(\bar x,\bar z) = d(\bar y,\bar t)$.

To prove the lemma, we set $r=\ell_X(h)$ and choose $\bar x,\bar y,\bar z,\bar t\in \mathbf{H}^2$ to be a \emph{sub-embedding} in the sense of~\cite[II.1.10]{bh} for the following four points: $x$, $y=hx$, $z$ the projection of $x$ to $L_h$ and $t=hz$, which is he projection of $y$ to $L_h$. The definition of a sub-embedding is that one should have
$$d(x,y)=d(\bar x,\bar y), \kern2mm d(z,t)=d(\bar z,\bar t), \kern2mm d(x,z)=d(\bar x,\bar z), \kern2mm d(y, t)=d(\bar y,\bar t)$$
and
$$d(x, t) \leq d(\bar x, \bar t), \kern2mm d(y, z) \leq d(\bar y, \bar z).$$
In particular, since the Alexandrov angles $\angle_{z}(x,t)$ and $\angle_{t}(y,z)$ are~$\geq \pi/2$ by a property of projections, we conclude that the points $\bar x,\bar y,\bar z,\bar t$ are as required above; the lemma follows.
\end{proof}

\begin{proof}[Proof of Proposition~\ref{prop:xi_cont}]
Fix $g\in H$ and $p_g \in L_g$. For any $h\in H$, we shall denote by $p_h$ the projection of $p_g$ to $L_h$; thus $h^n p_h$ travels on $L_h$ towards $\xi^+_h$. Fix some $D>0$. In view of the definition of the cone topology on $\partial X$, it suffices to prove the following claim: for any $n\in \mathbf{N}$, there is a neighborhood $U$ of $g$ in its level-set in $H$ such that $d(p_g, p_h)$ and $d(g^n p_g, h^n p_h)$ are both at most $D$.

The function $h\mapsto d(h p_g, p_g)$ is continuous; therefore, Lemma~\ref{lem:comp} shows that we can find $U$ such that $d(p_g, p_h)<D/2$ for all $h\in U$. We further restrict $U$ by requiring $d(g(g^j p_g), h(g^j p_g))<D/2n$ for all $j=0, \ldots, n-1$. This implies
$$d(h^{n-(j+1)} g^{j+1} p_g, h^{n-j} g^j p_g) \ < \ \frac{D}{2n} \kern5mm\forall\, j=0, \ldots, n-1$$
and hence $d(g^n p_g, h^n p_g)< D/2$. We conclude $d(g^n p_g, h^n p_h)< D$ and the claim is proved.
\end{proof}

Given a subset $A\se X$ of a geodesic metric space $X$, let $[A]$ be its closed convex hull. We recall that $X$ is \emph{non-positively curved in the sense of Busemann}~\cite{busemann48} if the distance function is convex (in particular the space is uniquely geodesic); see~\cite[II.1.18]{bh}.

\begin{lemma}\label{lemma:buse}
Suppose that $X$ is non-positively curved in the sense of Busemann. Then the operation $A\mapsto [A]$ is $1$-Lipschitz for the Hausdorff distance on bounded non-empty sets $A\se X$.
\end{lemma}

\begin{proof}
For any subset $A$, let $A'$ be the collection of midpoints of all pairs of points in $A$. We thus obtain a sequence $A^{(0)}=A$, $A^{(n+1)} =(A^{(n)})' \supseteq A^{(n)}$. Moreover, $[A]$ coincides with the closure of the union of all $A^{(n)}$. Therefore, it suffices to show that the operation $A\mapsto A'$ is $1$-Lipschitz for the Hausdorff distance on bounded non-empty sets. This follows from the definition of convexity.
\end{proof}

Before going back to the study of the $C_{t}$, we will need one more general fact about our strong topology, which will be used in the proof of Theorem~\ref{continuous}. Recall that an action $\pi$ is called (metrically) \emph{proper} if the function $d(\pi(g)x, x)$ tends to infinity as $g$ leaves compact subsets of $G$ (for some, equivalently any, $x\in X$). 

\begin{prop}\label{coradsemicont}
The quantity $\corad_X(\pi(G)(*))$ is upper semi-continuous at $(X, *, \pi)$ on the class of all pointed metric $G$-spaces if $\pi$ is proper.
\end{prop}

\begin{proof}
Let $(X, *, \pi)$ and $\delta>0$ be given; write $r=\corad_X(\pi(G)(*))$, which we can assume finite since otherwise the proposition is void. We choose $0<\epsilon\leq\delta/4$, $R>2 r+ 6\epsilon$ and a compact set $M\se G$ so large that $g*\in B_{X}(*, R)$ implies $g\in M$. This is possible since $\pi$ is proper.

Suppose for a contradiction that $(X', *', \pi')$ is in $U_{R, \epsilon, M}(X, *, \pi)$ but contains a point $y$ with $d(\pi'(G)(*'), y) \geq r+\delta$. Upon translating, we can moreover assume $d(*', y) < d(\pi'(G)(*'), y) + \epsilon$. Let $z$ be the point at distance $r+4\epsilon$ from~$*'$ on a geodesic connecting~$*'$ to $y$; this point exists since $4\epsilon \leq \delta$.

Let $f\colon X'\to X$ be as in the definition of $U_{R, \epsilon, M}(X, *, \pi)$. Since $d(z, *')=r+4\epsilon$, the point $z$ is in $B_{X'}(*', R)$ and we have $d(f(z), *)<r+5\epsilon$. Let $g\in G$ be such that $d(g *, f(z))<r+\epsilon$. Notice that $g*$ belongs to the ball $B_{X}(*, R)$ by the choice of $R$; it follows that $g\in M$. Now we have
\begin{multline*}
d(g *', z) \ <\ d(f(g *'), f(z)) + \epsilon \ <\  d(g f(*'), f(z)) + 2 \epsilon\\
=\ d(g *, f(z)) + 2 \epsilon \ <\ r+3\epsilon.
\end{multline*}
It follows
\begin{multline*}
d(g *', y) \ \leq\ d(y,z) + d(g *', z)  \ <\ d(y,z) +r+3\epsilon \\
=\ d(*', y) - d(*', z)+r+3\epsilon \ =\ d(*', y) - \epsilon \ <\ d(\pi'(G)(*'), y),
\end{multline*}
which is impossible. This contradiction finishes the proof.
\end{proof}


\subsection{Continuity of the family of actions}\label{preuvecontinuity}

Having introduced our strong topology in the previous subsection, we now want to prove the continuity of the family 
$$t\longmapsto (C_{t},\ast, \varrho_{t})$$
for $t\in (0,1]$. Before establishing this, we shall need:
\begin{itemize}
\item first, to see all the convex sets $C_{t}$ as subsets of a fixed infinite dimensional hyperbolic space, in other words we will need to chose some identification between the spaces $\HHI_{t}$ and the model $\HHI$ described in the introduction (recall from Section~\ref{qfhs} that the definition of $\HHI_{t}$ involved a certain completion depending on $t$);
\item second, to check carefully the continuity at $t=1$ since the representation $\pi_{s}$ degenerates to a reducible one as $t$ goes to $1$. 
\end{itemize}

\noindent

We will address these two points by proving: 

\begin{prop}\label{identification_rho}
There exists a family $(\widetilde{\varrho}_{t})_{0<t\le 1}$of representations of $\Isomn$ into the isometry group $\Isomi$ of $\HHI$ with the following properties:
\begin{enumerate}
\item For $t\in (0,1)$, $\widetilde{\varrho}_{t}$ is conjugated to the action of $\Isomn$ on $\HHI_{t}$ given by $\varrho_{t}$.\label{conjoriginal}
\item The action $\widetilde{\varrho}_{1}$ preserves an $n$-dimensional totally geodesic subspace (on which the action is conjugated to the standard $\Isomn$-action on $\HH$).\label{boutintervalle} 
\item For each $p\in \HHI$, the map  $(t,g)\mapsto \widetilde{\varrho}_{t}(g)(p)$ is continuous on $(0,1]\times \Isomn$.\label{supercontinuity} 
\item The unique $\widetilde{\varrho}_{t}(K)$-fixed point does not depend on $t\in (0,1)$; it is contained in the invariant $n$-dimensional totally geodesic subspace for $\widetilde{\varrho}_{1}$.\label{etoilenebougepas}
\end{enumerate}
\end{prop}

Abusing notations, we will still denote by $\ast$ the unique $\widetilde{\varrho}_{t}(K)$-fixed point ($0<t<1$) in $\HHI$. We postpone the proof of Proposition~\ref{identification_rho}, and first prove Theorem~\ref{continuous} assuming the proposition.

\medskip

For $t\in (0,1)$, let $i_{t}\colon  \HH \to \HHI$ be the unique $\rhonew_{t}$-equivariant map from $\HH$ to $\HHI$. We will entertain a slight abuse of notation and still denote by $C_{t}$ the smallest invariant convex set for $\rhonew_{t}$ ($0<t<1$). Finally we denote by $C_{1}\subset \HHI$ the $\rhonew_1(\Isomn)$-invariant totally geodesic subspace and by $i_1 : \HH \to C_1\subset \HHI$ the unique $\rhonew_1(\Isomn)$-equivariant map.

\begin{lemma}\label{continuityboundary}
The map
$$\begin{array}{rll}
(0,1] \times \partial \HH & \lra \partial \HHI \\
(t,\xi) & \longmapsto i_{t}(\xi) \\
\end{array}$$
\noindent is continuous.
\end{lemma}

\begin{proof}
Since, for a fixed $t$, the image of $i_{t}$ is equal to the orbit of a point of $\HHI$ under $\rhonew_{t}(K)$, it is enough to prove that for a fixed $\xi \in \partial \HH$, the map $t\mapsto i_{t}(\xi)$ is continuous. Choose a one-parameter semi-group $\{g_r\}_{r\geq 0}$ of hyperbolic elements of $\Isomn$ whose attracting fixed point in $\partial \HH$ is $\xi$, normalized by $\ell_{\HH}(g_r)=r$. According to Proposition~\ref{identification_rho}\eqref{supercontinuity}, $\rhonew_{t}(g_{1/t})$ is pointwise continuous in $t\in (0, 1]$. Moreover, since $\varrho_{t}(g_{1/t})$ has constant translation length, so does $\rhonew_{t}(g_{1/t})$. Therefore, Proposition~\ref{prop:xi_cont} ensures that $t\mapsto i_{t}(\xi)$ is continuous because $i_{t}(\xi)$ is the attracting fixed point of $\rhonew_{t}(g_{1/t})$.
\end{proof}

We are finally ready for the main result of this subsection.

\begin{prop}\label{enfincasetermine}
The map $(0,1]\ni t \mapsto C_{t}$ is continuous for the strong topology on pointed metric $\Isomn$-spaces.
\end{prop}

\begin{proof}
Fix $t_0\in (0, 1]$, $R>0$, $\epsilon >0$ and a non-empty compact subset $M$ of $\Isomn$. Let $D>0$ be large enough so that $\rhonew_{t}(M) (B_{\HHI}(*, R)) \se {B_{\HHI}(*, R+D)}$ whenever $|t-t_{0}|<\epsilon$. We work in the Klein model for $\HHI$, hence $\HHI \cup \partial \HHI$ is identified with the closed unit ball of a Hilbert space $\mathscr{H}$ in such a way that $\ast$ coincides with the origin of $\sH$. Notice that hyperbolic balls centered at~$\ast$ correspond to Euclidean balls centered at the origin (and of radius~$<1$). Moreover, the hyperbolic and Euclidean distances on $B_{\HHI}(\ast, R+D)$ are bi-Lipschitz to each other, so that we can argue with the Euclidean one for the remainder of the proof.

For any $t\in (0, 1]$ we denote by $f\colon  C_t\to C_{t_0}$ the nearest point projection (in $\sH$) and choose this map for Definition~\ref{def:topol}. In view of the continuity of $\rhonew_{t}(g)$ in $t$ and $g$, the proposition will follow if we prove that when $t$ is close enough to $t_0$, the Hausdorff distance between $B_{C_t}(*, R+D)$ and $B_{C_{t_0}}(*, R+D)$ is less than, say,~$\epsilon/2$ (note that to check condition~\eqref{presqueequivariance} from Definition~\ref{def:topol}, one can use the local compactness of $C_{t_{0}}$ since the maps of Proposition~\ref{identification_rho}\eqref{supercontinuity} have not been proved to be uniformly continuous in $p$).

Therefore, it suffices to establish the continuity of the map $t\mapsto C_t$ when $C_t$ is considered as a bounded subset of $\mathscr{H}$; equivalently, we can replace $C_t$ by its closure $\overline{C_t}$. At this point we recall that $\overline{C_t}$ can be seen as the (affine) closed convex hull of $i_t(\partial \HH)$, see Section~\ref{sec:Ct:def}. Since $\partial \HH$ is compact, Lemma~\ref{continuityboundary} implies that the map $t\mapsto  i_{t}(\partial \HH)$ is continuous on $(0, 1]$ for the Hausdorff topology on subsets of the Hilbert sphere. Therefore we are done by Lemma~\ref{lemma:buse}.
\end{proof}

At this point we have concluded the proof of Theorem~\ref{continuous}, except for the fact that $\corad_{C_{t}}(i_{t}(\HH))$ converges to $0$ as $t$ converges to $1$. But this follows from Proposition~\ref{coradsemicont}. 

\bigskip

We now turn to the proof of Proposition~\ref{identification_rho}. From now on, we will be constantly using both the parameter $t$ and the parameter $s$. We thus remind the reader that they are related by the equation $t=(n-1)(s-\frac{1}{2})$. To keep track of the parameter $t$ when dealing with the various bilinear forms we have at hand, we will denote by $B_{t}$ the symmetric bilinear form $\langle \cdot , \cdot \rangle_{s}$ introduced in Section~\ref{prelim-princeseries}. Finally, we will denote by $\vert \cdot \vert$ the ${\rm L}^{2}$-norm on $\LLB$ and if $v\in \LLB$ by $v_{k}$ the component of $v$ in the subspace $H^{k}\subset \LLB$. 

\medskip

We start with the following:

\begin{lemma}\label{OOO} Let $v\in \LLB$. 
\begin{itemize}
\item The map $(s,g) \mapsto \pi_{s}(g)(v)$ is continuous on $\R_{+}^{\ast}\times \Isomn$.
\item For each $s_{\ast}\in \R_{+}^{\ast}$, one can write: $\pi_{s}(g)(v)=\pi_{s_{\ast}}(g)(v)+(s-s_{\ast})\iota(s,g,v)$ where $\iota(s,g,v)$ is continuous on $(s_0,s_1]\times \Isomn$. In particular, $\pi_{s}(g)(v)$ is differentiable with respect to $s$.  
\end{itemize}
\end{lemma}
\begin{proof} Let $m(s):=\pi_{s}(g)(v)$. We can write:
$$m(s)=u(s,g)\cdot w(g),$$
where $w$ is the continuous function $g\mapsto \vert {\rm Jac}(g^{-1})\vert^{\frac{1}{2}} v\circ g^{-1}\in \LLB$ and $u(s,g)=\vert {\rm Jac}(g^{-1})\vert^{s}$. It therefore suffices to prove all the results of the lemma for $u$, {\it thought of as a map with values in ${\rm L}^{\infty}(\partial \HH)$} since the bilinear map
$$\begin{array}{ccc}
{\rm L}^{\infty}(\partial \HH) \times \LLB & \lra & \LLB \\
(v_1,v_2) & \longmapsto & v_1 \cdot v_2 .\\
\end{array}$$
is continuous. But now the statement of the lemma follows from the fact $u(s,g)(b)$ is smooth simultaneously in all its variables. This completes the proof.\end{proof}

To study the continuity at $t=1$, we will need the following proposition. Recall that the special values $s_{j}$ of the parameter $s$ were defined by $s_{j}=\frac{j}{n-1}+\frac{1}{2}$; they correspond to the value $j$ of the parameter $t$. 

\begin{prop}\label{reducibility} The representation $\pi_{s_{1}}$ preserves the space $V_{2}:=\bigoplus_{k=2}^{\infty}H^{k}$ as well as the inner product $U_{2}$ on $V_{2}$ defined by: 
$$U_{2}(u,v)=\underset{t\to 1}{{\rm lim}}\frac{-B_{t}(u,v)}{1-t}.$$
\end{prop}

There is of course a similar result replacing $s_{1}$ by $s_{0}$ or by $s_{j}$ for some $j$ greater or equal to $2$, see~\cite{johnsonwallach}. In the course of the proof we will need the following observation. The forms $B_{t}$ satisfy (for $t\in (0,1)$ and $v\in \LLB$):
\begin{equation}\label{majoration}
\vert B_{t}(v,v)\vert \le \vert v\vert^{2}.
\end{equation} 
Indeed, given our normalization in~\eqref{eq:Ls}, all the coefficients of the intertwining operators $L_{s}$ are bounded by $1$, see~\eqref{valeurlambda}.

\begin{proof}
We already know that the form $B_{t}$ is invariant by $\pi_{s}$ for all $s>0$, even if $\pi_{s}$ is reducible. Hence $\pi_{s_{1}}$ must preserve the kernel of the form $B_{1}$, which is nothing else than the space $V_{2}$. Using Equation~(\ref{valeurlambda}), one checks that the bilinear form $U_{2}$ on $V_{2}$ defined by the limit above satisfies:
$$U_{2}(v,v)=\frac{1}{n(n+1)}\sum_{k\ge 2}\prod_{j=2}^{k-1}\frac{j-1}{j+n}\vert v_{k}\vert^{2}.$$
We now show that it is invariant by the (restriction of) the representation $\pi_{s_{1}}$ i.e. that 
$$U_{2}(v,v)=U_{2}(\pi_{s_{1}}(g)(v),\pi_{s_{1}}(g)(v)),$$
for $v\in V_{2}$ and $g\in \Isomn$. Write
$$\pi_{s}(g)(v)=a(t)+c(t)$$
where $a(t)\in H^{0}\oplus H^{1}$ and $c(t)\in V_{2}$. Note that $a(t)\to 0$ and $c(t)\to c(1)=\pi_{s_{1}}(g)(v)$ as $t\to 1$. Since the curve $t\mapsto a(t)$ is differentiable (Lemma~\ref{OOO}) and since $a(1)$ is $0$, we have $a(t)=O(1-t)$ as $t$ goes to $1$. Since $\pi_{s}(g)$ preserves $B_{t}$, we have:
\begin{equation}\label{formebilineairelimite}
\frac{-B_{t}(v,v)}{1-t}=\frac{-B_{t}(a(t),a(t))}{1-t}+\frac{-B_{t}(c(t),c(t))}{1-t}.
\end{equation}
According to~(\ref{majoration}), the first term is bounded by $(1-t)^{-1}\vert a(t)\vert^{2}=O(1-t)$ which converges to $0$. The second term is equal to
\begin{equation}\label{seriedelicate}
\frac{-B_{t}(c(t),c(t))}{1-t}=\frac{t}{(t+n-1)(t+n)}\sum_{k\ge 2}\left(\prod_{j=2}^{k-1}\frac{j-t}{j+t+n-1}\right)\vert c_{k}(t)\vert^{2}
\end{equation}
where we have written $c_{k}(t)$ for the component of $c(t)$ in $H^{k}$ (for $k\ge 2$). Let us consider the ``tail" of the above series:
$$S^{N,t}:=\frac{t}{(t+n-1)(t+n)}\sum_{k\ge N}\left(\prod_{j=2}^{k-1}\frac{j-t}{j+t+n-1}\right)\vert c_{k}(t)\vert^{2}.$$
We have $$\vert S^{N,t}\vert \le \sum_{k\ge N}\vert c_{k}(t)\vert^{2}\underset{t\to 1}{\longrightarrow}\sum_{k\ge N}\vert c_{k}(1)\vert^{2}$$
since $c(t)\to c(1)$. Hence $\limsup_{t\to 1} S^{N,t}$ can be made as small as we wish by taking $N$ large. This implies that one can take the limit term by term in~(\ref{seriedelicate}) and hence $-B_{t}(c(t),c(t))/(1-t)$ converges to $U_{2}(\pi_{s_{1}}(g)(v),\pi_{s_{1}}(g)(v))$. Summing up all pieces, we have that the left-hand side of~(\ref{formebilineairelimite}) converges to $U_{2}(v,v)$ and its right-hand side converges to $U_{2}(c(1),c(1))=U_{2}(\pi_{s_{1}}(g)(v),\pi_{s_{1}}(g)(v))$.\end{proof}

Each operator $\pi_{s_{1}}(g)\colon \LLB \to \LLB$ can be represented by a matrix associated to the decomposition 
$$\LLB = (H^{0}\oplus H^{1})\oplus V_{2}$$ 
which has the following form:
\begin{equation}\label{reptriangulaire}
\left(\begin{array}{cc}
A(g) & 0\\
C(g) & D(g)\\
\end{array}\right).
\end{equation}
From now on, we will denote by $\widetilde{\pi}_{s_{1}}$ the representation of $\Isomn$ on $\LLB$ defined by $\pinew_{s_{1}}(g)=A(g)\oplus D(g)$; it preserves the quadratic form of hyperbolic type $U_{1}\oplus -U_{2}$ where $U_{1}$ is the restriction of $B_{1}$ to $H^{0}\oplus H^{1}$. 

\smallskip

Now, one would like to transform the representation $\pi_{s}$, to make it converge to the diagonal representation $\widetilde{\pi}_{s_{1}}$ rather than to the triangular one $\pi_{s_{1}}$. We shall make the off-diagonal block $C(g)$ in~(\ref{reptriangulaire}) vanish as $t\to 1$ by a conjugation which depends on $t$. For that purpose we introduce the map $u_{t} \colon  \LLB \to \LLB$ defined by \mbox{$u_t= \Id\oplus\sqrt{1-t}\Id$} with respect to the decomposition $(H^{0}\oplus H^{1})\oplus V_{2}$. Then, for each $s\in (s_{0},s_{1})$, we define the new representation $\pinew_{s}$ of $\Isomn$ on $\LLB$ by: 
$$\pinew_{s}(g):=u_{t}\circ \pi_{s}(g)\circ u_{t}^{-1}.$$ 
A consequence of the next proposition is that $\pinew_s$ converges to $\pinew_{s_{1}}$ as $s$ goes to $s_1$. 

\begin{prop}\label{commentchoisirunnom} For each vector $v\in \LLB$ the map 
$$(s,g) \mapsto \pinew_{s}(g)(v)$$ 
is continuous on $(s_0,s_1]\times \Isomn$.
The representation $\pinew_{s}$ preserves the symmetric bilinear form $\BNEW_{t}$ defined by:
\begin{align*}
\BNEW_{t}(v,v)&=B_{t}(u_{t}^{-1}(v),u_{t}^{-1}(v)) \kern 7mm \text{if $t\in (0,1)$},\\
\BNEW_{1}&=U_{1}\oplus -U_{2}.
\end{align*}
\end{prop}

To make the statement of the proposition more transparent, let us describe more concretely the bilinear forms $\BNEW_{t}$. The decomposition
$$\LLB = \bigoplus_{k=0}^{\infty} H^k$$
is orthogonal for all these forms; on $H^0$ and $H^1$, the form $\BNEW_t$ coincides with $B_t$; on $H^k$ with $k\ge 2$, the form $\BNEW_t$ is equal to 
$$\frac{-t}{(t+n-1)(t+n)}\prod_{j=2}^{k-1}\frac{j-t}{j+t+n-1}$$
times the standard ${\rm L}^{2}$ scalar product.

\begin{proof}
We will decompose the operator $\pi_{s}(g)$ block-wise according to the decomposition $\LLB=\left( H^{0}\oplus H^{1}\right) \oplus V_{2}$. We write
$$\pi_{s}(g)=\left(
\begin{array}{cc}
A_{s}(g) & N_{s}(g) \\
C_{s}(g) & D_{s}(g) \\
\end{array}\right),$$
where $A_{s}(g)$ maps $H^0\oplus H^1$ into itself, $N_{s}(g)$ maps $V_2$ into $H^0\oplus H^1$, etc. If $v\in \LLB$, we write $v=(u,w)$ where the first coordinate is in $H^0\oplus H^1$ and the second in $V_2$. Using this notation one finds (for $t<1$)
$$\pinew_{s}(g)(u,w)=(A_{s}(g)(u)+\frac{N_s(g)(w)}{\sqrt{1-t}}, \sqrt{1-t}C_{s}(g)(u)+D_{s}(g)(w)).$$
According to Lemma~\ref{OOO}, $\pi_{s}(g)(u,w)$ is a continuous function of $(s,g)$, hence $A_{s}(g)(u)$, $\sqrt{1-t}C_{s}(g)(u)$, $\frac{N_s(g)(w)}{\sqrt{1-t}}$ and $D_{s}(g)(w)$ are continuous functions of $(s,g)$ on $(s_0,s_1)\times \Isomn$. For the continuity at points of the form $(s_1,g_{1})$, one simply has to deal with the term 
$$\frac{N_s(g)(w)}{\sqrt{1-t}}.$$
We decompose the operator $v=u+w\mapsto \iota(s,g,v)$ appearing in Lemma~\ref{OOO} block-wise:
$$\iota(s,g,u+w)=(\iota_1(s,g,u)+\iota_2(s,g,w),\iota_{3}(s,g,u)+\iota_4(s,g,w)).$$
Lemma~\ref{OOO} guarantees that $\frac{N_s(g)(w)}{\sqrt{1-t}}=-\frac{\sqrt{1-t}}{n-1}\iota_2(s,g,w)$, which converges to $0$ as $(s,g)$ converges to $(s_1,g_1)$. This proves the statement about the continuity of $\pinew_{s}(g)$. The fact that $\pinew_{s}$ preserves $\BNEW_{t}$ is clear. 
\end{proof}

\medskip

Let $B$ be the bilinear form on $\LLB$ defined by $B(v,v)=\vert v_{0}\vert^{2}-\sum_{k=1}^{\infty}\vert v_{k}\vert^{2}$. We will denote by $W\subset \LLB$ the dense subspace of vectors $v$ such that $v_{k}=0$ for all but finitely many $k$. For $t\in (0,1]$, we choose a $K$-equivariant operator $\varphi^{t}$ of ${\rm L}^{2}(\partial \HH)$ such that 
$$B(\varphi^{t}(v),\varphi^{t}(v))=\BNEW_{t}(v,v) \;\;\;\;\; (v\in \LLB)$$
($\varphi^{t}$ is unique up to a choice of sign on each $H^{k}$). From now on, we will denote by $\HHI$ the hyperbolic space constructed from the form $B$.  Observe that $\varphi^t$ is not invertible (this follows for instance from Lemma~\ref{tendverszero}). However $(\varphi^{t})^{-1}$ is well-defined on the subspace $W$ defined above. Hence for each $g\in \Isomn$, the map

$$\varphi^{t}\circ \pinew_{s}(g) \circ (\varphi^{t})^{-1} : W \to \LLB$$
is well-defined and preserves $B$.  One can thus consider the restriction 
\begin{equation}\label{avantderniereequation}
\varphi^{t}\circ \pinew_{s}(g) \circ (\varphi^{t})^{-1}  : W\cap \HHI \to \HHI
\end{equation}
This map being an isometric embedding (and $W\cap \HHI$ being dense in $\HHI$) it admits an extension to all of $\HHI$. We denote by $\rhonew_{t}(g)$ this extension. The map 
$$\rhonew_{t} : \Isomn \to \Isomi$$
is obviously a homomorphism. We now claim that the family $(\rhonew_{t})_{0<t\le 1}$ satisfies all the conditions of Proposition~\ref{identification_rho}. 
\begin{itemize}
\item For $t\in (0,1)$ the map $\varphi^{t}\circ u_t$ is defined on a dense subset of the hyperboloid $\HHI_{t}$ constructed from the form $B_{t}$ and is an isometric embedding into $\HHI$. It thus extends to an isometry $\HHI_{t}\to \HHI$ which intertwines $\varrho_t$ and $\rhonew_t$. This proves~\eqref{conjoriginal}. 
\item A similar argument shows that $\rhonew_1$ is conjugated by $\varphi^1$ to the action given by $\pinew_{s_{1}}$ on the hyperboloid constructed from $\widetilde{B}_{1}$. Since $\pinew_{s_{1}}$ is reducible, we obtain~\eqref{boutintervalle}. 
\item We have to check the continuity of the map $(t,g)\mapsto \rhonew_{t}(g)(p)$ for any $p\in \HHI$. Since the maps $\rhonew_{t}(g)$ are all isometric, it suffices to do so for a dense set of $p$, for instance for $p\in W\cap \HHI$. Now, the topology on $\HHI$ is the one induced by $\LLB$ hence it is enough to prove the continuity of the map 
$$(t,g)\mapsto \rhonew_{t}(g)(p)$$ considered as a map with values into $\LLB$, for each $p\in W\cap \HHI$. By linearity we can assume that $p$ lies in a single $H^k$. In that case $(\varphi^{t})^{-1}$ is just the multiplication by a constant depending on $t$, hence it is enough to prove that $(t,g)\mapsto \varphi^{t}(\pinew_{s}(g)(p))$ is continuous. According to Proposition~\ref{commentchoisirunnom}, $\pinew_{s}(g)(p)$ is continuous as a function of $(s,g)$. Using the fact $\varphi^t$ is uniformly bounded for $t\in (0,1]$, one proves as in Proposition~\ref{reducibility} that $\varphi^t(\pinew_{s}(g)(p))$ is continuous in $(t,g)$ (controlling separately the first $N$ coordinates of $\varphi^t(\pinew_{s}(g)(p))$ and its tail). This proves~\eqref{supercontinuity}.
\item The constant function $\boldsymbol{1}_{\partial\HH} \in H^0\subset \LLB$ is fixed by $\rhonew_t(K)$ for all $t$, thus proving~\eqref{etoilenebougepas}. 
\end{itemize}

This concludes the proof of Proposition~\ref{identification_rho}.\qed

\subsection{Towards a renormalization limit}\label{renor}

In this paragraph we make a few remarks concerning the behavior of the spaces $C_{t}$ and of the actions $\varrho_{t}$ as $t$ converges to $0$. The next proposition shows that the right scale to renormalize the family $C_{t}$ as $t$ goes to $0$ is $\sqrt{t}$. 

\begin{prop}\label{prop:renorm}
Fix any $g\in\Isomn$. Then

$$\lim_{t\to 0}\ \frac{d(\pi_{s}(g)(\ast),\ast)}{\sqrt{t}}$$

\noindent exists and is positive unless $g$ is in $\mathrm{O}(n)$.
\end{prop}

We start with an elementary observation.

\begin{lemma}\label{lemma:jac}
Let $g$ be a smooth diffeomorphism of a compact Riemannian manifold $M$ and $n>1$. We endow $M$ with its normalized volume form. If $g$ does not preserve the Riemannian measure of $M$, then the function

$$t\longmapsto \int_M \vert \jac(g)\vert^\frac{n-1+t}{n-1}$$

\noindent has a Taylor series around $t=0$ starting with $1 + a_1 t + \cdots$ where $a_1>0$.
\end{lemma}

\begin{proof}[Proof of the lemma]
We only need to show that the derivative at $t=0$ is positive, recalling $\int_M \vert \jac(g)\vert=1$. Differentiating under the integral sign, that derivative is

$$\frac1{n-1} \int_M \vert \jac(g)\vert \log\big(\vert \jac(g)\vert \big)\ = \ -\frac1{n-1} \int_M \log\big(\vert \jac(g^{-1})\vert \big),$$

\noindent where the second expression is obtained by change of variable. The equality case of Jensen's inequality shows that the latter expression is positive unless $| \jac(g^{-1})|$ is identically~$1$.
\end{proof}

In other words, $a_1$ measures the Kullback--Leibler divergence between $|\jac(g)|$ and $1$.

\begin{proof}[Proof of Proposition~\ref{prop:renorm}]
By the polar decomposition, any element $g$ can be written $g=k g_u k'$ where $k, k'\in \mathrm{O}(n)$ and $g_u=g_{e^{u},0,{\rm Id}}$ as before. Thus it suffices to consider the case of $g_u$ (and we can assume $u>0$). The formula~(\ref{eq:distance}) of Section~\ref{princeseries} give
$$\cosh d(\pi_{s}(g_u)(\ast),\ast) = \int_{\bS^{n-1}} \jac(g_u)^\frac{n-1+t}{n-1}.$$
Therefore, Lemma~\ref{lemma:jac} combined with the Taylor series of the hyperbolic cosine concludes the proof.
\end{proof}

The representation $\pi_{s_{0}}$ on $\LLB$ was defined by \mbox{$\pi_{s_{0}}(g)\cdot f=|{\rm Jac}(g^{-1})| f\circ g^{-1}$.} As shown by a change of variable argument, it preserves the subspace $V_{1}\subset \LLB$ of functions with mean $0$.  The representation $\pi_{s_{0}}$ also preserves the inner product $U_{1}$ on $V_1$ defined by:
$$U_{1}(v,v)=\underset{t\to 0}{{\rm lim}}\frac{-B_{t}(v,v)}{t}.\;\;\;\;\; (v\in V_{1})$$
The proof of this fact is similar to the proof of Proposition~\ref{reducibility}. The map $c\colon \Isomn \to V_{1}$ defined by 
$$c(g)=\pi_{s_{0}}(g)(1)-1,$$
(where $1$ stands for the constant function equal to $1$) is a cocycle, since it is a formal coboundary. Let $\varrho_{0}$ be the associated affine action:
$$\varrho_{0}(g)(v)=\pi_{s_{0}}(g)(v)+c(g).$$
Note that this action is considered in~\cite[chap. 3,  \S\, 3.4]{ccjjv}. However, there the authors are interested in the relation between the family of unitary representations $\pi_{s}$ for $s\in(0,s_{0})$ (the complementary series) and $\pi_{s_{0}}$. Here, we approach $\pi_{s_{0}}$ from the other side of the looking glass. It would be interesting to determine whether the $\Isomn$-spaces $\frac{1}{\sqrt{t}}C_{t}$ admit a limit as $t$ goes to $0$ and if they do, to relate the limit to the affine action $\varrho_{0}$ just described.

\bigskip
 A related problem is to determine minimal invariant convex sets for affine actions on Hilbert spaces. It is asked in~\cite{ctv} (see Remark 4.9 there) whether there exists an isometric action of a semisimple Lie group on a (non-zero, real) Hilbert space $\sH$ for which the closed convex hull of any orbit is equal to all of $\sH$. In that context, we shall prove the following dichotomy. (Untill the end of section~\ref{renor}, Hilbert spaces are real, except in Corollary~\ref{cor:Gelfand}.)

\begin{prop}\label{rayon}
Let $G$ be a topological group with a continuous isometric action on a Hilbert space $\mathscr{H}$. Assume that its linear part is irreducible and without $K$-invariants for some compact subgroup $K$. Then any closed $G$-invariant convex set either is $\mathscr{H}$ or has empty geometric boundary.
\end{prop}

(The geometric boundary is still understood in the \cat0 sense~\cite[II.8]{bh}; in the present case it is non-empty if and only if the convex subset of $\mathscr{H}$ contains a half-line.)

\medskip

Recall that a \emph{Gelfand pair} $(G, K)$ consists of a locally compact group $G$ with a compact subgroup $K$ such that the convolution algebra $\sA$ of compactly supported bi-$K$-invariant continuous functions on $G$ is commutative. This algebra is called the \emph{Hecke algebra}. Examples include all simple connected center-free Lie groups $G$ with a maximal compact subgroup $K$; the criterion originally due to Gelfand also applies to $G=\Isomn$ and $K=\mathrm{O}(n)$, see e.g.~\cite{faraut}.

\begin{cor}\label{cor:Gelfand}
Let $(G,K)$ be a Gelfand pair with $G$ compactly generated. Consider a continuous isometric action of $G$ on a complex Hilbert space $\sH$ with nontrivial linear part. Assume that the linear part is irreducible over $\R$. If $G$ contains an element with positive translation length, then $\sH$ is the only non-empty $G$-invariant closed convex set.   
\end{cor}
\begin{proof}[Proof of the corollary]
Let $\alpha$ be the action, $\pi$ be its linear part, $b$ the associated cocycle and $g\in G$ an element with positive translation length. The existence of $g$ implies that $G$ has no fixed point, i.e. the action is non-trivial in cohomology. By~\cite[Prop.~V.3 p.~306]{delorme} this implies that $\pi(K)$ has no invariant vectors. Let $C$ be a closed convex invariant set. If $C$ is non-empty, we can assume that $0\in C$. In view of Proposition~\ref{rayon}, it suffices to prove that the segment $[0,\alpha(g^{n})(0)]$ converges to a ray in $C$. The assumption on $g$ implies that $|\alpha(g^{n})(0)|/n$ converges to a positive number. Since
$$\alpha(g^{n})(0)=\sum_{j=0}^{n-1}\pi(g)^{j}(b(g)),$$
the von Neumann ergodic theorem implies that $\alpha(g^{n})(0)/n$ converges to a non-zero vector.\end{proof}

\begin{rem} The argument in the proof above, combined with Moore's theorem~\cite{moore}, implies that if $G$ is a simple Lie group, all elements have translation length $0$ for any isometric action of $G$ on a Hilbert space.  Indeed if $g$ had positive translation length, $\alpha(g^{n})(0)/n$ would converge to a non-zero $\pi(g)$-invariant vector; by Moore's theorem this vector is $\pi(G)$-invariant, but $G$ being perfect implies that $b$ projects trivially on the subspace of $\pi(G)$-invariants.
\end{rem}

We will need the following lemma, which shows in particular that for compact groups, cyclic vectors are ``positively cyclic'' in the absence of invariant vectors. The assumption on invariant vectors is of course necessary. The compactness of the group is also necessary, as shown by the example of the regular representation of an infinite discrete group.

\begin{lemma}\label{prop:chinoise}
Let $\pi$ be any continuous linear representation of a compact group $K$ on a (Hausdorff) locally convex topological vector space $V$ over $\R$. If $V^K=0$, then the closed convex cone generated by the $K$-orbit of any $v\in V$ coincides with the $K$-cyclic sub-representation generated by~$v$.
\end{lemma}

\begin{proof}[Proof of the lemma]
It suffices to show that the closed convex cone generated by $\pi(K)(v)$ contains $-v$. Let thus $U$ be an arbitrarily small convex neighborhood of $-v$. There is an open neighborhood $A$ of the identity in $K$ such that $\pi(g)(v)\in -U$ for all $g\in A$. Let $\mu$ be the normalized Haar measure of $K$. We claim that
\begin{equation}\label{eq:chinoise}
\frac1{\mu(A)} \int_{K\setminus A} \pi(g) (v)\, d\mu(g)
\end{equation}
belongs to $U$, in which case the proof is complete.

\smallskip
To prove the claim, observe that $\int_K \pi(g) (v)\, d\mu(g)$ is $K$-invariant and hence vanishes. Therefore, the expression in~(\ref{eq:chinoise}) coincides with $-\frac1{\mu(A)} \int_A \pi(g) (v)\, d\mu(g)$. Thus the claim follows from the choice of $A$ and the convexity of $-U$.
\end{proof}

\begin{proof}[Proof of Proposition~\ref{rayon}] We denote by $\alpha$ the isometric action on $\sH$ and write $$\alpha(g)(v)=\pi(g)(v)+c(g)$$ for $v\in \sH$. Consider a closed invariant convex set $C$. We must prove that if $C$ contains a half-line, then $C=\mathscr{H}$. Let $c(t)=x_{0}+tv$ ($t\ge 0$) be a half-line in $C$ and assume by contradiction that $C\neq \mathscr{H}$. Then, there exists a non-zero linear form $\varphi$ on $\mathscr{H}$ and a constant $a$ such that
$$C\subset \{x\in \mathscr{H}, \varphi(x)\ge a\}.$$
Since $c(t)$ is in $C$, we must have $\varphi(v)\ge 0$. We can apply the same result to any vector of the form $\pi(k)(v)$ for $k\in K$ since $C$ contains the ray $\alpha(k)(c(t))$. In particular we have $\varphi(\pi(k)(v))\ge 0$ for all $k$ in $K$ hence $\varphi(u)\ge 0$ for any vector $u$ which is a linear combination with positive coefficients of the vectors $\{\pi(k)(v)\}_{k\in K}$. According to Proposition~\ref{prop:chinoise} we must have $\varphi(-v)\ge 0$ as well, hence $\varphi(v)=0$. Hence the geometric boundary of $C$ is contained in the boundary of a proper closed linear subspace of $\mathscr{H}$. This is impossible since $\pi$ is irreducible. Therefore $C=\mathscr{H}$.\end{proof}


\section{Further results}\label{further}

\subsection{Complements on trees}\label{furthertree}

One of the early motivations for this article was the analogy with representations of automorphism groups of trees. We recall in particular the following result from~\cite{bim}, in which it is implicitly assumed that the tree has no leaf (i.e.\ degree-one vertex); we shall keep this as a standing assumption.

\smallskip
\begin{center}
\begin{minipage}[t]{0.85\linewidth}
\itshape
For any simplicial tree $\mathscr T$ and every $\lambda>1$ there is a representation $\pi_\lambda\colon \mathrm{Aut}(\mathscr{T})\to \Isomi$ and a $\pi_\lambda$-equivariant map $\Psi_\lambda\colon\mathscr{T}\to \HHI$ extending continuously to the boundary with
$$\cosh d(\Psi_\lambda(x), \Psi_\lambda(y)) = \lambda^{d(x,y)}$$
for any vertices $x,y\in\mathscr{T}$. Moreover, $\Psi_\lambda(\mathscr{T})$ has finite coradius in the closed convex hull of $\Psi_\lambda(\partial\mathscr{T})$ in $\HHI$.
\end{minipage}
\end{center}

\smallskip
The argument of Section~\ref{sec:Ct:def} allows us to show that for locally finite trees, the conclusion can be strengthened. Just as for $\Isomn$, we obtain ``exotic'' proper \cat{-1} spaces:

\begin{prop}\label{prop:tree}
If the simplicial tree $\mathscr{T}$ is locally finite, then the closed convex hull $C$ of $\Psi_\lambda(\partial\mathscr{T})$ in $\HHI$ is locally compact. Thus, if $\mathrm{Aut}(\mathscr{T})$ acts cocompactly on $\mathscr{T}$, then (via $\pi_\lambda$) it acts cocompacty on $C$.
\end{prop}

Thus, for instance, we deduce that the finitely presented torsion-free simple groups constructed in~\cite{burger-mozes} appear as ``convex-cocompact'' subgroups of isometries of $\HHI\times\HHI$.

\begin{proof}[Proof of Proposition~\ref{prop:tree}]
The second statement follows from the first since $\Psi_\lambda(\mathscr{T})$ has finite coradius in $C$. As for the first, it follows as in Section~\ref{sec:Ct:def} by applying the Mazur compactness theorem in the Klein Model.
\end{proof}

By the result of~\cite{bim} quoted above, the actions of  $\mathrm{Aut}(\mathscr{T})$ on $\HHI$ given by $\pi_\lambda$ are non-elementary unless the $\mathrm{Aut}(\mathscr{T})$-action on $\mathscr{T}$ is itself elementary. Thus, we can also consider the canonical convex subspace $C_\lambda\se \HHI$ provided by Lemma~\ref{lemma:C_exists}. This subset is contained in the convex hull $C$ of Proposition~\ref{prop:tree}, for instance because Lemma~\ref{lemma:C_exists} can be applied to both $ \HHI$ and $C$. However, in contrast to Lemma~\ref{lemma:all_C}, it seems that the minimal subset could be smaller than $C$ in general. At any rate, even for regular trees, the initial representation constructed in~\cite{bim} is far from irreducible before passing to a smaller subspace, and there is an infinite-dimensional subspace of $K$-fixed vectors, denoting by $K$ the stabilizer of a vertex.

\subsection{\texorpdfstring{Actions on the symmetric space of ${\rm O}(p,\infty)$}{Actions on other symmetric spaces}}\label{higherrank}
Let $\mathscr{H}$ be a separable real Hilbert space endowed with a basis $(e_{i})_{i\ge 1}$. In what follows, we will denote by ${\rm O}(p,\infty)$ the group of linear operators from $\mathscr{H}$ to itself which preserve the symmetric bilinear form $B_{p}$ defined by 
$$B_{p}\left(\sum_{i\ge 1}x_{i}e_{i},\sum_{i\ge 1}x_{i}e_{i}\right)=x_{1}^{2}+\cdots +x_{p}^{2}-\sum_{i\ge p+1}x_{i}^{2}.$$
We also denote by $X(p,\infty)$ the space of all $p$-dimensional subspaces $V\subset \sH$ on which the form $B_{p}$ is positive definite. The space $X(p,\infty)$ can be endowed with an ${\rm O}(p,\infty)$-invariant distance which turns it into an infinite dimensional symmetric space; see~\cite{duchesne,duchesne2} for an introduction to this space. In this section, we collect a few remarks concerning isometric actions of $\Isomn$ on $X(p,\infty)$.

\medskip
As we have seen in Section~\ref{princeseries} , the classical study of intertwiners for the spherical principal series provides a whole lot of irreducible representations
$$\Isomn \lra  {\rm O}(p,\infty),\kern5mm p=\binom{n-1+ j}{n-1}\kern5mm \text{where}\ j=0, 1, 2, \ldots$$

\begin{prob}\label{prob:ranks}
Show that the above representations exhaust all possible irreducible representations $\Isomn \to {\rm O}(p,\infty)$.
\end{prob}

This would in particular restrict the possible ranks~$p$ when $n\geq 3$, whilst we have already observed that for $n=2$, every rank $p\in\mathbf{N}$ can occur. We propose the following first evidence for Problem~\ref{prob:ranks}, where we restrict to representations of the identity component $\Isomn^{\circ}$ of $\Isomn$.

\begin{main}\label{thm:gelf}
Let $p$ be an integer with $2<p<n$, where $n>4$. Then there is no irreducible representation $\Isomn^{\circ} \to {\rm O}(p,\infty)$.
\end{main}

Of course, the first subquestion of Problem~\ref{prob:ranks} left open by the theorem above is the existence of irreducible representations 
$$\Isomn^{\circ} \lra {\rm O}(2,\infty)$$
when $n\ge 3$.

\bigskip

The analogue of the following proposition for \emph{unitary} representations is well-known. It cannot possibly hold for arbitrary representations on a Hilbert space, but finite index sesquilinear forms provide just enough spectral theory to obtain the conclusion. (We recalled the definition of Gelfand pairs just before Corollary~\ref{cor:Gelfand} above.)

\begin{prop}\label{prop:schur}
Let $(G, K)$ be a Gelfand pair and $k=\R$ or $\C$. Let $\pi$ be a continuous $k$-linear representation of $G$ on a $k$-Hilbert space $\sH$ preserving a continuous, strongly non-degenerate (sesqui)linear form of finite index. If $\pi$ is irreducible, then the space $\sH^K$ of $K$-invariant vectors has $k$-dimension at most~$1$ if $k=\C$ and at most~$2$ if $k=\R$.
\end{prop}

\noindent
As always, irreducibility means that there is no \emph{closed} $G$-invariant proper subspace. The \emph{strong} non-degeneracy refers to a completeness condition, see~\cite{bim}.

\begin{proof}[Proof of Proposition~\ref{prop:schur}]
We choose a Haar measure on $G$ (necessarily bi-invariant since $G$ must be unimodular~\cite[24.8.1]{Simonnet}) and we denote by $\sA$ the convolution algebra of compactly supported bi-$K$-invariant continuous functions on $G$. We consider the more complicated case $k=\R$; it will be clear from the proof that it applies to $k=\C$, with minor simplifications. Let $B$ be the bilinear form under consideration.

For any continuous linear operator $T$ of $\sH$, we denote by $T^\dagger$ its $B$-adjoint, i.e. the unique continuous linear operator such that $B(Tu, v) = B(u, T^\dagger v)$ holds for all $u,v\in \sH$. The $B$-adjoint exists: it can be produced as $T^\dagger = J^{-1}T^*J$ if $J=J^*$ is the operator with $B(x,y)=\langle J x, y\rangle$. By assumption, we have
\begin{equation}\label{eq:adj}
\pi(g)^\dagger\ =\ \pi(g^{-1}) \kern5mm \forall\, g\in G.
\end{equation}
We can extend $\pi$ to a representation of the convolution algebra $\mathrm{C}_c(G)$ of all compactly supported continuous functions on $G$ by integration against the Haar measure. Let $\mu_K$ be the normalized Haar measure of $K$, seen as a measure on $G$. Then $\mu_K$ is a convolution idempotent and
$$\mu_K * \mathrm{C}_c(G) * \mu_K \ =\ \sA\ =\ \mu_K * \sA * \mu_K.$$
Further, we can define the continuous linear operator $P:=\pi(\mu_K)$, yielding a $K$-invariant operator $P\colon\sH\to\sH^K$ which is the identity on $\sH^K$. In particular, the space $\sH^K$ is $\sA$-invariant.

We claim that $B$ is non-degenerate on $\sH^K$, thus turning it itself into a quadratic space of finite index, say $q\geq 0$. Indeed, suppose that $v$ is a non-zero element of $\sH^K$. There is $w\in\sH$ with $B(v, w)\neq 0$. Using~(\ref{eq:adj}) and the fact that $K$ is unimodular, being compact, we can make the change of variable $k\mapsto k^{-1}$ in the equation $P=\int_{K}\pi(k)dk$ and find $P^\dagger = P$. Thus $B(v, w) = B(P v, w) = B(v, P w)$, which proves the claim.

Next, in view of the definition of $\sA$ and of the fact that a Gelfand pair is unimodular, we observe that~(\ref{eq:adj}) also implies that $\sA$ is \emph{symmetric} in the sense that $a^\dagger \in\sA$ for all $a\in \sA$. A result of Naimark (Corollary~2 in~\cite{naimark1963}) states that a commutative symmetric algebra of operators of a complex Hilbert space endowed with a sesquilinear form of finite index~$q$ preserves a (non-negative) subspace of dimension~$q$. Applying this to the $\sA\otimes\C$-representation on the complexification $\sH^K_\C$ endowed with $B_\C$, we deduce that $\sA$ preserves a finite-dimensional subspace of $\sH^K$. At that point, the classical finite-dimensional Schur lemma applies and we find a subspace $L\se \sH^K$ of dimension one or two invariant under $\sA$ (unless $\sH^K=0$, in which case we are done). Since $\pi$ is irreducible, the space $W$ spanned by $\mathrm{C}_c(G)\cdot L$ is dense in $\sH$. Therefore, $P(W) = \sA\cdot L \se L$ shows that $L= \sH^K$, finishing the proof.
\end{proof}

\begin{proof}[Proof of Theorem~\ref{thm:gelf}]
 Since $K^{\circ}={\rm SO}(n)$ is compact, it must fix a point in the symmetric space ${\rm X}(p,\infty)$. This means that there is a positive definite, $p$-dimensional subspace $V\se \mathscr{H}$ which is $K^{\circ}$-invariant. The smallest dimension of a non-trivial representation of $K^{\circ}$ being $n$ when $n\neq 4$, the hypothesis $p<n$ implies that the action of $K^{\circ}$ on $V$ is trivial. Since $p>2$, this implies that the space $\mathscr{H}^{K^{\circ}}$ has dimension greater than~$2$, and the representation $\pi$ cannot be irreducible according to Proposition~\ref{prop:schur}.
\end{proof}

We conclude with a last proposition. We will say that a representation $\varrho \colon G \to {\rm O}(p,\infty)$ is {\it geometrically Zariski dense} if $\varrho(G)$ does not fix any point at infinity in $X(p,\infty)$ and does not preserve any non-trivial closed totally geodesic sub-manifold. For finite dimensional symmetric spaces, this is equivalent to the Zariski density in the usual sense. When $p=1$, $X(1,\infty)=\HHI$ and any totally geodesic submanifold is the intersection of $\HHI$ with a linear subspace of the corresponding Hilbert space (in the hyperboloid model). In that case a representation $\varrho$ is irreducible if and only if it is geometrically Zariski dense. There is no such correspondence when $\HHI$ is replaced by $X(p,\infty)$ ($p\ge 2$). However, one still has: 

\begin{prop} Let $\varrho \colon  G \to {\rm O}(p,\infty)$.
\begin{itemize}
\item If $\varrho$ is geometrically Zariski dense, then $\varrho$ is irreducible.
\item If $\varrho$ is irreducible, $\varrho(G)$ does not fix a point in the boundary of $X(p,\infty)$.
\end{itemize}
\end{prop}

\begin{proof}
We first assume that $\varrho$ is geometrically Zariski dense and prove its irreducibility. Let $V\subset \mathscr{H}$ be a closed non-zero invariant subspace. We assume by contradiction that $V\neq \mathscr{H}$. If the restriction of the quadratic form $B_{p}$ to $V$ is non-degenerate, then $\mathscr{H}=V\oplus V^{\perp}$ (see~\cite{bim}). Since $G$ has no fixed point in ${\rm X}(p,\infty)$, we can assume, upon replacing $V$ by its orthogonal, that $B_{p}$ is not definite on $V$. Hence it has signature $(q,q')$ for some positive integers $q\le p$ and $q'\le +\infty$. If $B_{p}$ is definite on $V^{\perp}$, this implies that $G$ preserve a totally geodesically embedded copy of the symmetric space $X(q,q')$. If $B_{p}$ is not definite on $V^{\perp}$, $G$ preserves a product of two symmetric subspaces $X(q,q')\times X(r,r')$, where $X(r,r')$ is associated to $V^{\perp}$. In any case, this is a contradiction. If on the other hand the restriction of $B_{p}$ to $V$ is degenerate, we can assume that $V$ is isotropic upon replacing it by $V\cap V^\perp$. According to Proposition~6.1 in~\cite{duchesne2}, this implies that $G$ preserves a point in the boundary at infinity of $X(p,\infty)$. This is a contradiction again.

\medskip
The second part of the proposition follows once again immediately from Proposition~6.1 in~\cite{duchesne2}.\end{proof}


\end{document}